\documentclass[11pt]{amsart}

\usepackage{amsmath,amssymb,amsthm,mathtools,bm,mathrsfs,graphicx}
\usepackage{bbm}
\usepackage[mathscr]{eucal} 

\usepackage[T1]{fontenc}
\usepackage[a4paper,margin=1in]{geometry}
\usepackage{indentfirst}
\usepackage{esint}
\usepackage{graphicx}
\usepackage{float}

\usepackage[colorlinks=true,linkcolor=blue,citecolor=blue,urlcolor=blue]{hyperref}
\newcommand{\dashint}{\fint}

\theoremstyle{plain}
\newtheorem{theorem}{Theorem}[section]
\newtheorem{lemma}[theorem]{Lemma}
\newtheorem{proposition}[theorem]{Proposition}

\theoremstyle{remark}
\newtheorem{remark}{Remark}[section]

\newcommand{\cB}{\mathcal{B}}  
\newcommand{\cC}{\mathcal{C}}  
\newcommand{\cA}{\mathcal{A}}
\newcommand{\cQ}{\mathcal{Q}}
\newcommand{\cS}{\mathcal{S}}
\newcommand{\cH}{\mathcal{H}}
\newcommand{\cM}{\mathcal{M}}
\newcommand{\cP}{\mathcal{P}}

\newcommand{\lan}{\langle}
\newcommand{\ran}{\rangle}

\newcommand{\intsigma}{{\mathring{\Sigma}}}
\newcommand{\oH}{\mathring{\mathrm{H}}^1}
\newcommand{\rb}{\mathrm{b}}

\allowdisplaybreaks
\newcommand{\bga}{\bm{\gamma}}
\newcommand{\pa}{\partial}
\newcommand{\dd}{\,\mathbf{d}}

\newcommand{\Si}{\Sigma}

\newcommand{\B}{\mathbb{B}}
\newcommand{\R}{\mathbb{R}}
\newcommand{\N}{\mathbb{N}}

\newcommand{\id}{\mathrm{id}}

\begin{document}
	
	\title[On Solutions for Singular Toda System on Surfaces with Boundary]
	{On Solutions for Singular Toda System \\ on Surfaces with Boundary}

	\author{Zhengni Hu}
	\address[Zhengni Hu]{School of Mathematical Sciences, Shanghai Jiao Tong University
		800 Dongchuan RD, Minhang District, 200240 Shanghai, China}
	\email{zhengni\_hu@sjtu.edu.cn}

	\author{Aleks Jevnikar}
	\address[Aleks Jevnikar]{Department of Mathematics, Computer Science and Physics,
		University of Udine, Via delle Scienze 206, 33100 Udine, Italy}
	\email{aleks.jevnikar@uniud.it}
	
	\author{Sang-Hyuck Moon}
	\address[Sang-Hyuck Moon] {Department of Mathematics and Institute of Mathematical Sciences, Pusan National University, Busan 46241, Republic of Korea}
	\email{shmoon@pusan.ac.kr}
	
	\date{}


	\begin{abstract}
		This paper studies a singular Toda system with linear source terms on a compact surface $ \Sigma $ with  boundary $ \partial\Sigma$. The aim is twofold. First, using isothermal coordinates and reflection arguments near the boundary, we perform blow-up analysis obtaining a compactness result. Second, by assuming that $\partial\Sigma$ has at least two connected components, we prove a general existence result via a min-max scheme combined with a retraction argument. 
	\end{abstract}

		
		\maketitle

	\section{Introduction}
	In this paper, we consider the following singular $SU(3)$  Toda system on a compact surface $(\Sigma,g)$ with smooth boundary $\partial\Sigma$ and unit volume $|\Sigma|=1$: 
	\begin{equation}\label{eq:toda_0}
		\begin{aligned}
			\begin{cases}
				-\Delta_g v_1 + \beta v_1 =\sum^2\limits_{j=1} a_{1j} \rho_j  \left( \frac{h_j e^{v_j}}{\int_{\Sigma} h_j e^{v_j} \, dv_g} - 1\right) - \sum^{\iota}\limits_{l=1} \varrho(p_l)\gamma_{1l}\left(\delta_{p_l}-1\right) &\text{in } \intsigma,\\
				-\Delta_g v_2 + \beta v_2 =\sum^2\limits_{j=1} a_{2j} \rho_j  \left( \frac{h_j e^{v_j}}{\int_{\Sigma}h_j e^{v_j}  dv_g} - 1\right) - \sum^{\iota}\limits_{l=1} \varrho(p_l)\gamma_{2l}\left(\delta_{p_l}-1\right)  &\text{in } \intsigma,\\
				\partial_{\nu_g} v_1 = \partial_{\nu_g} v_2 = 0\hfill &\text{on }  \partial \Sigma.
			\end{cases}
		\end{aligned}
	\end{equation}
	Here, $\intsigma$ is the interior of $\Sigma$,  $\Delta_g$ is the Laplace-Beltrami operator, $dv_g$ is the area element of $(\Sigma,g)$,  $\nu_g$ is the unit outward normal to the boundary $\partial \Sigma$,  $\delta_p$ is the Dirac measure centered at $p\in \Sigma$, $(a_{ij})=\left(\begin{matrix}
		2 &-1\\
		-1 &2
	\end{matrix}\right)$ is the Cartan matrix,  $\rho_1, \rho_2>0$ are parameters,  $\beta\in \R$ is the coefficient of the linear source terms,   $h_1, h_2$ are smooth positive potential functions, {$\mathcal{Q}:=\{p_l\}_{l=1}^\iota\subset \Sigma$ is the singular point set}, $\gamma_{il}>-1$ is the weight of the singular point for $i=1,2 \text{ and }l=1,\cdots,\iota$, and $\varrho(x)=
	4\pi $ for $x\in \intsigma, $ $ \varrho(x)=
	2\pi$ for $ x\in \partial \Sigma$. 
	
	Assume that $-\beta$ is not an eigenvalue   of $-\Delta_g$ restricted to the zero-mean Neumann space. For each $p\in\Sigma$, let $G_p^g:=G^g(\cdot,p)$ be the Green function satisfying
	\begin{equation}\label{Green-F-G}
		\begin{aligned}
			\begin{cases}
				-\Delta_g  G^g(\cdot,p)+\beta G^g(\cdot,p)= \delta_p -1 & \text{ in } \intsigma,\\
				\partial_{\nu_g} G^g(\cdot, p)= 0 & \text{ on } \partial\Sigma,\\
				\int_{\Sigma} G^g(\cdot,p) =0. &
			\end{cases}
		\end{aligned}
	\end{equation}
	The existence and estimates for the Neumann Green function in the case $\beta=0$, including boundary poles, are established in \cite[Lemma~6]{yang_wang2021}; see also \cite{ABH2026} for  a related construction for $-\Delta_g+\beta$. 
	The properties of $G_p^g$ needed in the present setting are collected in Lemmas \ref{lem:Green-properties} and \ref{cor:Green-potential}.
	\par

	Define  
	\begin{equation*}
		u_i= v_i + \sum_{l=1}^\iota  \varrho(p_l) \gamma_{il}G^g(\cdot,p_l),\quad V_i= h_i e^{- \sum_{l=1}^\iota  \varrho(p_l)\gamma_{il}G^g(\cdot,p_l)}.
	\end{equation*}
	Under this change of variables, the Toda system~\eqref{eq:toda_0} is transformed into the following system: 
	\begin{equation}\label{eq:toda}
		\begin{aligned}
			\begin{cases}
				-\Delta_g u_1 + \beta u_1=\sum\limits_{j=1}^2 a_{1j} \rho_j  \left( \frac{V_j e^{u_j}}{\int_{\Sigma} V_j e^{u_j} \, dv_g} - 1\right)  & \text{ in } \, \intsigma,\\
				-\Delta_g u_2 + \beta u_2 =\sum\limits_{j=1}^2 a_{2j} \rho_j  \left( \frac{V_j e^{u_j}}{\int_{\Sigma} V_j e^{u_j} \, dv_g} - 1\right)  &  \text{ in } \, \intsigma,\\
				\partial_{\nu_g} u_1 = \partial_{\nu_g} u_2 = 0   \,
				&\text{ on }\, \partial \Sigma,
			\end{cases}
		\end{aligned}
	\end{equation}
	where $V_j(x)>0$ {for any $x\in \Sigma\setminus\mathcal{Q}$}, and $V_j(x)\simeq d^{2\gamma_{il}}_g(x,p_l)$ {as $x\to\,p_l$} for any $j=1,2;\, l=1,\cdots,\iota$. 
	\par

	Define a subspace of $H^1(\Sigma)$ with zero average, $$\oH:=\left\{ v\in H^1(\Sigma): \int_{\Sigma} v dv_g =0\right\},$$
	equipped with the inner product $\lan u,v\ran_{\oH}=\int_{\Sigma} \lan\nabla u,\nabla v\ran_g dv_g$ for any $u,v\in\oH$. We will study the singular Toda system \eqref{eq:toda} in
	$\cH:= \oH\times \oH,$ a subspace of $H^1(\Sigma) \times H^1(\Sigma)$ with  $ \int_{\Sigma} u_i \, dv_g = 0 $ for  $i = 1, 2$. 
	
	The system \eqref{eq:toda} is variational and has a corresponding Euler-Lagrange functional  for $(u_1, u_2)\in \cH$:
	\begin{equation*}
		\begin{split}
			J_{\beta,\rho}(u_1,u_2)=\int_{\Sigma} Q(u_1,u_2)dv_g +\beta\int_{\Sigma} Q_0(u_1,u_2) dv_g- \sum_{i=1}^2 \rho_i\log \int_{\Sigma} V_i e^{u_i} dv_g,
		\end{split}
	\end{equation*}
	where $\rho=(\rho_1,\rho_2), $ the quadratic forms $$ Q(u_1,u_2)=\frac 1 3 \left( |\nabla u_1|_g^2+ |\nabla u_2|_g^2 + \langle \nabla u_1, \nabla u_2 \rangle_g \right), $$ 
	and   $Q_0(u_1,u_2)= \frac 13 (u_1^2+u_2^2+u_1u_2)$
	for  $u=(u_1,u_2) \in \cH.$
	
	The linear term $\beta u_i$ is motivated, in particular, by its scalar counterpart. For mean field equations, a linear term of this type arises naturally in the study of stationary Keller–Segel systems; see, for instance, \cite{Battaglia2018,wang_wei2002_steady}. This motivates us to consider the corresponding linear perturbation of the $SU(3)$ Toda system.
	
	The system \eqref{eq:toda} with $\beta = 0$ and positive potentials $V_i > 0$ for $i = 1, 2$ corresponds to the classical Toda system. This system arises in various contexts in geometry and physics, including the study of holomorphic curves in $\mathbb{C}\mathbb{P}^2$ \cite{bolton1988conformal,doliwa1997holomorphic, lin2012classification}, and in non-abelian Chern-Simons vortex theory \cite{dunne1991self, dunne1995self, nolasco1999double, nolasco2000vortex, yang1999relativistic, yang2001solitons}. 
	
	The existence of critical points of $J_{0,\rho}$ with positive potentials has been extensively studied on closed surfaces and bounded domain in $\R^2$ with Dirichlet conditions. 
	By the classical Moser-Trudinger inequality,
	it is well known that  $J_{0,\rho}$ is bounded from below if and only if $\rho_i \leq 4\pi$ for each $i = 1, 2$; see~\cite{jostandwang2001}. It follows that $\rho_i < 4\pi$ for both $i=1,2$, the functional $J_{0,\rho}$ becomes coercive and thus admits a minimizer via the standard variational method.
	
	However, when $\rho_i$ exceeds $4\pi$, $J_{0,\rho}$ is no longer bounded from below, rendering the standard variational method inapplicable. By some geometric conditions, solutions for parameters at the coercive threshold can be obtained.
	Jost, Lin and Wang demonstrated in \cite{jost_lin_wang2006} that $J_{0,\rho}$ has a minimizer when $(\rho_1,\rho_2)$ in $\{ 4\pi\}\times (0, 4\pi)$ or $ (0, 4\pi)\times \{ 4\pi\}$ for closed surfaces. Independently, Li explores the scenario where $(\rho_1, \rho_2) = (4\pi, 4\pi)$  in \cite{li2005solutions}.  
	Subsequently, by applying a min-max scheme  with a refined Moser-Trudinger inequality, significant progress has been made in  the existence of solutions to the classical $SU(3)$ Toda system.
	In general genus settings, Malchiodi and Ndiaye demonstrate in \cite{Malchiodi2007SomeER} that the Toda system on closed surfaces has solutions for $(\rho_1, \rho_2)$ in $ (0, 4\pi) \times(4\pi m, 4\pi(m + 1))$ or  $(4\pi m, 4\pi(m + 1))\times (0, 4\pi)$ with non-negative integer $m$, by a min-max scheme. 
	Jevnikar, Kallel, and Malchiodi further consider cases in which one of the $\rho_i$ values lies in the interval $(4\pi, 8\pi)$ in~\cite{jevnikar_kallel_Malchiodi2015}.
	When the genus of surface $\Sigma$ is positive, solutions for the Toda system on closed surfaces (including the singular Toda system with non-negative singular weights) for $\rho_i \notin 4\pi\mathbb{N}$ for $i=1,2$, have been addressed in \cite{battaglia_jevnikar_ruiz2015} by a min-max scheme. 
	Despite the above results, the existence of solutions for the Toda system with general $\rho=(\rho_1,\rho_2)$ values remains an open question, even for closed surfaces. For related work on singular Toda systems on closed surfaces with negative singular weights, see \cite{battaglia_existence_2015}.
	
	When $\Sigma$ is a surface with boundary, from \cite{yang2006extremal} the Moser-Trudinger inequality turns to be 
	\begin{equation}
		\label{eq:MT-b} \log \int_{\Sigma} e^u dv_g \leq 
		\frac 1 {8\pi} \int_{\Sigma} |\nabla u|_g^2 dv_g  +C, 
	\end{equation}
	for all $u\in \oH$ and some constant $C>0.$
	The functional $J_{0,\rho}$ is coercive when $\rho_1, \rho_2 \in (0, 2\pi)$, which implies the existence of a minimizer via standard variational methods. In this setting, the threshold is $2\pi$ instead of $4\pi$ because blow-up points may occur on the boundary. In the critical case, Zhu \cite{zhu_solutions_2011} studied the existence of a minimizer for $J_{0,\rho}$ with $(\rho_1,\rho_2) = (2\pi, 2\pi),$ under certain geometric assumptions.

	To obtain the existence, it is necessary to establish the global compactness, which is of independent interest, as it is a necessary tool for obtaining solutions via variational methods, see e.g. \cite{battaglia_jevnikar_ruiz2015,Malchiodi2007SomeER,jevnikar_kallel_Malchiodi2015} and also plays an important role in the degree counting for the Toda system, see e.g. \cite{lee_degree_2018,Lee2020,GuZhang2020}.
	
	Assuming that $u^n=(u_1^n, u_2^n)$ is a sequence solutions of~\eqref{eq:toda} for $\rho^n\rightarrow \rho$, define the local limit masses as 
	\begin{equation}\label{def:masses}
		\sigma_i(x)= \lim_{{r\rightarrow 0}} \lim_{{n\rightarrow +\infty}}\int_{U_r(x)} \frac{ \rho^n_i  V_i e^{u^n_i}}{\int_{\Sigma}  V_i e^{u^n_i} dv_g}dv_g, \text{ for } i=1,2,
	\end{equation}
	and the blow-up set for $u^n_i$ as 
	\[ \mathcal{S}_i:=\left\{x\in \Sigma: \exists x_n\rightarrow x \text{ such that } \Big(u^n_i(x_n)- \log \int_\Sigma V_i  e^{u^n_i} dv_g\Big) \rightarrow +\infty \right\}. \]
	Here, $U_{r}(x):= \{ x'\in \Sigma: d_g(x,x')< r\}$ denotes geodesic balls.
	
	Let $\mathcal{S}:=\mathcal{S}_1\cup\mathcal{S}_2$, which is called the blow-up set of the Toda system~\eqref{eq:toda} for $\{ u^n\}_n$. For any $i=1,2$, let
	\begin{equation*}
		\begin{aligned}
			\gamma_i(x) = 
			\begin{cases}
				0, &\text{if } x \notin\mathcal{Q},\\
				\gamma_{ij}, &\text{if } x = p_j,\quad j=1,\cdots,\iota.
			\end{cases}
		\end{aligned}
	\end{equation*}
	\par 
	
	Let 
	\begin{equation}\label{eq:rela_local_mass}
		\sigma_1^2+\sigma_2^2-\sigma_1\sigma_2=\varrho\cdot(1+\alpha_1)\sigma_1+ \varrho\cdot(1+\alpha_2)\sigma_2.
	\end{equation}
	By a Pohozaev type identity, we have the following characterization of  the local limit masses at blow-up points (see Lemma \ref{lemma:identity}): for 
	$x\in\cS$,  $(\sigma_1(x),\sigma_2(x))$ is in
	\[  \{ (\sigma_1,\sigma_2):\sigma_1,\sigma_2 \ge 0 \text{ s.t. } \eqref{eq:rela_local_mass} \text{ holds for }\varrho=\varrho(x), \alpha_1=\gamma_1(x), \alpha_2=\gamma_2(x)\}.\]
	For $x\in \cS\setminus\cQ$, by the quantization of the local limit masses, $(\sigma_1(x), \sigma_2(x))$ is one of the following combinations: 
	\begin{eqnarray*}
		\big(0,\varrho(x)\big), \big(\varrho(x),
		0\big),  \big( \varrho(x),2\varrho(x)\big), 
		\big(2\varrho(x), \varrho(x)\big),\big(2\varrho(x),2\varrho(x)\big).
	\end{eqnarray*}
	For the case $\beta=0$ with empty boundary, the statement above is  given by \cite{jost_lin_wang2006}. Furthermore,  \cite{LinWeiZhang2015} proves that $(\sigma_1, \sigma_2)$ can only take a finite number of values for the Dirichlet problem. We will extend these results to our setting using an extension technique (see Section \ref{sec:2}). When the blow-up point lies on the boundary, the corresponding masses are half of those in the interior case. 
	
	Following the formulation in \cite{battaglia_jevnikar_ruiz2015}, for  any {$\left(\alpha_1, \alpha_2\right)\in (-1,+\infty)\times (-1,+\infty), x\in \Sigma$,} we define  
	\begin{equation*}
		\begin{split}
			\Gamma_{\alpha_1, \alpha_2}(x) :=\Big\{ (\sigma_1,\sigma_2): {\sigma_1,\sigma_2 \ge 0} \text{ s.t. } \eqref{eq:rela_local_mass} \text{ holds  for } \varrho=\varrho(x)  \Big\}	
		\end{split}
	\end{equation*}
	which is the subset of an ellipse in $\mathbb{R}^2$. 
	
	Then, we define 
	$\Lambda_{\alpha_1, \alpha_2}(x) \subseteq \Gamma_{\alpha_1, \alpha_2}(x)$,  constructed by the following rules:
	\begin{itemize}
		\item[1)]  $(0,0),$ $\big(\varrho(x)(1+\alpha_1), 0\big),$ $\big(0,\varrho(x)(1+\alpha_2)\big),$ $\big(\varrho(x)(1+\alpha_1), \varrho(x)(2+\alpha_1+\alpha_2)\big),$ $\big(\varrho(x)(2+\alpha_1+\alpha_2), \varrho(x) (1+\alpha_2)\big)$, and  $\big(\varrho(x)(2+\alpha_1+\alpha_2), \varrho(x)(2+\alpha_1+\alpha_2)\big)$ belong to $\Lambda_{\alpha_1, \alpha_2}(x)$;
		\item[2)] if $(a, b) \in \Lambda_{\alpha_1, \alpha_2}(x)$ then also any $(c, d) \in \Gamma_{\alpha_1, \alpha_2}(x)$ with $c=a+\varrho(x)m, m \in \N_0, d \geq b$ belongs to $\Lambda_{\alpha_1, \alpha_2}(x)$ where $\N_0$ is the set of non-negative integers;
		\item[3)]if $(a, b) \in \Lambda_{\alpha_1, \alpha_2}(x)$ then also any $(c, d) \in \Gamma_{\alpha_1, \alpha_2}(x)$ with $d=b+\varrho(x) n, n \in \N_0, c \geq a$ belongs to $\Lambda_{\alpha_1, \alpha_2}(x)$.
	\end{itemize}
	The local blow-up masses $(\sigma_1(x), \sigma_2(x))$ must lie in $\Lambda_{\alpha_1, \alpha_2}(x)$ for any $x \in \mathcal{S}$, see Section \ref{sec:2} and also \cite{LinWeiZhang2015}. Conversely, whether $(\sigma_1(x), \sigma_2(x))$ can attain all points in $\Lambda_{\alpha_1, \alpha_2}(x)$ remains an open question and  it relates to the existence of  blow-up solutions for the singular Toda system. For results on the classical Toda system, see \cite{musso_new_2016, DAprile_Pistoia_Ruiz2015,D'aprile_asymmetric_2016} and the references therein.
	For $i=1,2$, let
	{
		\begin{eqnarray*}
			\Gamma_i=\Big\{2\pi ( 2k_i+k'_i)+\sum_{x\in \cQ}\sigma_i(x) :  k_i,k'_i\in \N_0,  
			\sigma_i(x)\in \pi_i\left(\Lambda_{\gamma_1(x),\gamma_2(x)}(x)\right) \text{ for every } x\in \cQ\Big\},
		\end{eqnarray*}
	}
	where $\pi_i$ is the projection on the $i$-th component.  Then, we set
	$$
	\Gamma:=\left(\Gamma_1 \times [0, +\infty)\right) \cup\left([0,+\infty) \times \Gamma_2\right). $$
	
	\medskip
	
	\begin{theorem}~\label{thm:cpt_sols}
		Let  $\cA$ be a compact subset of $\mathbb{R}_{>0}^2\setminus \Gamma$, and assume that  $-\beta$ is not an eigenvalue   of $-\Delta_g$ restricted to the zero-mean Neumann space.  
		Then the solution space of \eqref{eq:toda} with $\rho=(\rho_1,\rho_2)\in \cA$ is compact in $\cH$. 
	\end{theorem}
	Compactness results and blow-up analysis for Toda systems on closed surfaces with $\beta = 0$ have been established in \cite{jost_lin_wang2006,Battaglia2015a_note,LinWeiZhang2015}. In contrast, Theorem \ref{thm:cpt_sols} treats compact  surfaces with boundary and includes the linear term $\beta u_i$. Moreover, the singular points are allowed to lie on $\partial\Sigma$. The boundary blow-up analysis is carried out using boundary isothermal coordinates and an even reflection, while the term involving $\beta$ enters as a lower-order term in the local analysis.
	
	We emphasize that, as far as the singular data are concerned, Theorem \ref{thm:cpt_sols} only requires $\gamma_{il}>-1$, and the singular points are allowed to lie on $\partial\Sigma$. The additional assumptions $\gamma_{il}\ge 0$ and $\cQ\subset\mathring{\Sigma}$ will be imposed only for the existence result below.
	
	\medskip
	
	We next turn to the existence problem. We approach it by a variational method, which is based on a min-max scheme in the spirit of Djadli and Malchiodi  \cite{DjadliMalchiodi2008} and then carried out for the system case in \cite{battaglia_jevnikar_ruiz2015}, studying the topology of low sublevels of $J_{\beta,\rho}$. By considering suitable retractions, taking care of the points weights (see the discussion below), we obtain the following general existence result.
	
	\medskip
	
	\begin{theorem}~\label{thm:main1}
		Suppose that 
	 $-\beta$ is not an eigenvalue   of $-\Delta_g$ restricted to the zero-mean Neumann space, $\gamma_{ij} \ge 0$ for all $i=1,2,j=1,\cdots, \iota$, and the singular point set $\cQ:=\{ p_l\}_{l=1}^{\iota} \subset \intsigma$. 
		If $\Sigma$ is a  surface  with  $\rb\geq 2$, where $\rb$ denotes the number of connected components of $\partial\Sigma$, then the problem~\eqref{eq:toda} has a solution for $\rho \in \R^2_{>0} \setminus \Gamma$.
	\end{theorem}
	
	One of the difficulties in the topological description of low sublevels is the interaction between $u_1$ and $u_2$, and the presence of singular points. Retractions on two disjoint curves not touching the singularities avoid these issues as in \cite{battaglia_jevnikar_ruiz2015}. However, a truly new difficulty is due to the effect of the boundary. Indeed, as will be evident in the sequel, concentration on the boundary is related to a different Moser-Trudinger inequality, see Lemma \ref{lem:local_MT}. Therefore, we are forced to distinguish between interior and boundary points, by considering their \emph{weights}. In the same spirit as \cite{car-mal} (where the distinction is between regular and singular points) the weight of an interior point will be $4\pi$, while for a boundary point $2\pi$, in accordance with the Moser-Trudinger constant. This mixed scenario makes the description of low sublevels very complicated. The crucial observation is that we can still map injectively the homology of the boundary concentration, where the weight is less compared to the interior part, into that of low sublevels. This can be done again by considering two retractions onto two curves $\bga_1, \bga_2\subset \pa\Sigma$. Here we exploit the assumption on  $\partial\Sigma$ having at least two connected components.
	
	\medskip
	
	The presence of the linear term $\beta u_i$ introduces an additional difficulty, as the operator $-\Delta_g + \beta$ need not be positive. 
	Let $0<\lambda_1\le\lambda_2\le\cdots$ be the positive Neumann eigenvalues of $-\Delta_g$ on the zero-mean space, counted with multiplicity, and set $\lambda_0=-\infty$.
	If $-\lambda_{I+1}<\beta<-\lambda_I$ for some $I \in \N_0$, the first $I$ eigendirections give rise to the negative part of the quadratic form. 
	Following the approach in \cite{Battaglia2018}, we incorporate these negative directions into the topology of the low sublevels through an $S^{I-1}$-factor.

	\

	The rest of the paper is organized as follows.
	In Section~\ref{sec:2}, we carry out a blow-up analysis for the singular 
	Toda system \eqref{eq:toda} on compact  surfaces with boundary. This leads to the quantization of local limiting masses and establishes the global compactness.
	In Section~\ref{sec:3}, we study the topological structure of the sublevel sets by introducing suitable barycenter spaces and retractions.
	Finally, in Section~\ref{sec:4}, we prove the existence result stated in Theorem~\ref{thm:main1}, using the min-max methods jointly with the compactness result. 
	
	\medskip
	
	{\bf Notations. }
	For any points $x,y\in \Sigma$, 
	$d_g(x,y)$ denotes the metric distance between $x$ and $y$ with respect to the Riemannian metric $g$. For any point $x\in \Sigma$ and subsets $\Omega, \Omega' \subseteq \Sigma$, 
	\[ d_g(x, \Omega):=\inf\{d_g(x,y): y\in \Omega\}\text{ and } d_g(\Omega,\Omega'):=\inf\{d_g(x,y): x\in \Omega, y\in \Omega'\}. \]
	We denote the open ball with radius $r$ centered at $x$ with respect to metric g as $U_{r}(x)$, the open ball centered at $x$ with radius $r$ in $\mathbb{R}^2$ as $\B_r(x)$ and $\B_r:=\B_r(0)$ and {$\B^+_r:=\B_r\cap \{y=(y_1,y_2)\in \R^2: y_2 > 0\}$, $I_r:=\{y\in\R^2: |y|<r,\ y_2=0\}$}. 
	
	For any $v\in L^1(\Sigma)$ and $\Omega\subseteq \Sigma$, we denote
	$|\Omega|_g:=\int_\Omega d v_g$, 
	$$\dashint _{\Omega} v:= \frac 1 {|\Omega|_g} \int_{\Omega} v dv_g, \qquad \text{ and }\qquad  \bar{v}:= \dashint_{\Omega} v d v_g.$$
	
	The sublevels of functional $J_{\beta,\rho}$ are defined as follows: 
	\[ J^a_{\beta,\rho}:=\left\{u=(u_1,u_2)\in \oH\times\oH: J_{\beta,\rho}(u_1,u_2)\leq a \right\}. \]
	
	{Given nonempty topological spaces $A_j \text{ for } j=1,\cdots,k$, we define the join $A_1\star \cdots\star A_k$ by
		\[ A_1\star \cdots \star A_k= (A_1\times \cdots\times A_k\times \Delta^{k-1})/\sim,\]
		where $\Delta^{k-1}=\{(t_1,\cdots,t_k) \in \R^k : \sum_{i=1}^k t_i=1, t_i \ge 0\}$ is the simplex and the equivalence relation $\sim$ is given by
		\[
		(a_1,\ldots,a_k;t_1,\ldots,t_k)\sim (a_1',\ldots,a_k';t_1,\ldots,t_k) 
		\]
		if $a_i=a_i'$  for every $i$  such that $t_i>0$.    In other words, the $i$-th coordinate is disregarded whenever $t_i=0$.
		For the joins used in the sequel, we adopt the convention that an empty factor is omitted; in particular,
		$X\star\emptyset=\emptyset\star X=X$.
		When we use simplex coordinates for all factors of the join, the coordinate associated with an empty factor is taken to be zero.}
	

	The letter $C>0$ will be a large constant that may vary among different formulas or even within the same line. Throughout the paper, we do not distinguish between sequences and subsequences.
	\par

	\section{Blow-up analysis}\label{sec:2}
	
	In this section, we study the compactness and blow-up behavior of solutions to \eqref{eq:toda}. 
	Our analysis is mainly based on \cite{jost_lin_wang2006,jostandwang2001,Lucia2002Vortex,LinWeiZhang2015}, where the corresponding problems are considered on planar domains or closed surfaces. 
	Since $\Sigma$ is a compact  surface with boundary, we use boundary isothermal coordinates as in \cite{yang_wang2021}, together with the reflection argument from \cite{wang_wei2002_steady}, to handle the homogeneous Neumann boundary condition.
	
	Before turning to the blow-up analysis, we collect some properties of the Green function that will be used throughout the paper.
	
	\begin{lemma}\label{lem:Green-properties}
		Assume that  $-\beta$ is not an eigenvalue   of $-\Delta_g$ restricted to the zero-mean Neumann space. For each $p\in\Sigma$, let $G_p^g=G^g(\cdot,p)$ be the Green function defined by \eqref{Green-F-G}. 
		More precisely, for every $p\in\Sigma$, including the case $p\in\partial\Sigma$, $G_p^g$ satisfies
		 \begin{equation*}
			\int_\Sigma \left( \langle\nabla_{g,x}G^g(x,p),\nabla_g\psi(x)\rangle_g +\beta G^g(x,p)\psi(x) \right)\,dv_g(x) = \psi(p)-\int_\Sigma\psi\,dv_g 
		\end{equation*} 
		for every $\psi\in C^\infty(\Sigma)$.
		
		Then, for each fixed $p\in\Sigma$,
		\[
		G^g(x,p)=-\frac{2}{\varrho(p)}\log d_g(x,p)+H_p^g(x) \qquad\text{as } \quad x\to p,
		\]
		where $H_p^g$ is bounded.
		Moreover, there exists $C>0$, independent of $x, p \in \Sigma$, such that $G^g(x,p)\geq-C$, and
		\[
		|G^g(x,p)|
		\leq C\bigl(1+|\log d_g(x,p)|\bigr),
		\qquad
		|\nabla_{g,x}G^g(x,p)|
		\leq \frac{C}{d_g(x,p)}
		\]
		for all $x,p\in\Sigma$ with $x\neq p$.
		Consequently, for every $q\in(1,2)$, there exists $C_q>0$ such that
		\[
		\sup_{p\in\Sigma} \|G_p^g\|_{W^{1,q}(\Sigma)} \leq C_q.
		\]
	\end{lemma}
	\begin{proof}
		Let $G_{0,p}^g$ be the normalized Neumann Green function associated with $-\Delta_g$. The corresponding logarithmic and gradient estimates, including the case $p\in\partial\Sigma$, are established in \cite[Lemma~6]{yang_wang2021}.
		
		Set $W_p:=G_p^g-G_{0,p}^g$.
		Then $W_p$ satisfies
		\[
		\begin{cases}
			(-\Delta_g+\beta)W_p=-\beta G_{0,p}^g
			&\text{ in }\intsigma,\\
			\partial_{\nu_g}W_p=0
			&\text{ on }\partial\Sigma.
		\end{cases}
		\]
		Since $\sup_{p\in\Sigma}\|G_{0,p}^g\|_{L^r(\Sigma)}\leq C_r$ for $r\in (1,\infty)$ and  $-\beta$ is not an eigenvalue   of $-\Delta_g$ restricted to the zero-mean Neumann space, the standard $W^{2,r}$-estimate for the Neumann problem yields $\sup_{p\in\Sigma}\|W_p\|_{W^{2,r}(\Sigma)}\leq C_r$.
		Choosing $r>2$, we obtain a uniform $C^{1,\alpha}$-bound for $W_p$. Hence, all the assertions follow from the corresponding properties of $G_{0,p}^g$. 
	\end{proof}
	
	As a corollary, we obtain the following Green representation formula.
	\begin{lemma}\label{cor:Green-potential}
		Let $f\in L^1(\Sigma)$ with $\int_\Sigma f dv_g=0$ and define
		\[
		u(x):=\int_\Sigma G^g(x,p)f(p) dv_g(p).
		\]
		Then, for every $q\in(1,2)$, $u\in W^{1,q}(\Sigma)$ and
		\[
		\|u\|_{W^{1,q}(\Sigma)} \leq C_q\|f\|_{L^1(\Sigma)}.
		\]
		Moreover, $u$ is the unique solution in $W^{1,q}(\Sigma)$ of 
		\[ \begin{cases} 
			-\Delta_g u+\beta u=f &\text{in }\intsigma, \\
			\partial_{\nu_g}u=0 &\text{on }\partial\Sigma,\\
			\int_\Sigma u dv_g=0, &
		\end{cases} \]
		in the weak sense, where the weak formulation is understood with test functions in $C^\infty(\Sigma)$.
	\end{lemma}
	\begin{proof}
		Fix $q\in(1,2)$ and define $Z(x):=\int_\Sigma f(p)\nabla_{g,x}G^g(x,p) dv_g(p)$. By Minkowski's integral inequality and Lemma \ref{lem:Green-properties},
		\[
		\|u\|_{L^q(\Sigma)}+ \|Z\|_{L^q(\Sigma)} \leq \int_\Sigma |f(p)| \big(\|G_p^g\|_{L^q(\Sigma)}+  \|\nabla_{g,x}G^g(\cdot,p)\|_{L^q(\Sigma)}\big) dv_g(p)\leq C_q \|f\|_{L^1(\Sigma)}.
		\]
		
		We claim that $Z=\nabla_g u$ in the sense of weak derivatives.
		Indeed, for every smooth vector field $X$ compactly supported in $\intsigma$, Fubini's theorem gives
		\[
		\begin{aligned}
			\int_\Sigma u \operatorname{div}_g X dv_g
			&=
			\int_\Sigma f(p)
			\left[
			\int_\Sigma
			G^g(x,p)\operatorname{div}_g X(x) dv_g(x)
			\right]dv_g(p)\\
			&=
			-\int_\Sigma
			\langle Z,X\rangle_g dv_g.
		\end{aligned}
		\]
		Therefore, $\nabla_g u=Z$ and $\|u\|_{W^{1,q}(\Sigma)} \leq    C_q\|f\|_{L^1(\Sigma)}$.
		
		By the defining property of the Green function, for every $p\in\Sigma$ and $\psi\in C^\infty(\Sigma)$, \[ \int_\Sigma \left( \left\langle \nabla_{g,x}G^g(x,p),\nabla_g\psi(x) \right\rangle_g + \beta G^g(x,p)\psi(x) \right) dv_g(x) = \psi(p)-\int_\Sigma\psi dv_g. \]
		Therefore, by Fubini's theorem and the assumption $\int_\Sigma f\,dv_g=0$, 
		\[ \begin{aligned} 
			\int_\Sigma \left( \langle\nabla_g u,\nabla_g\psi\rangle_g +\beta u\psi \right) dv_g &= \int_\Sigma f(p) \left( \psi(p)-\int_\Sigma\psi dv_g \right)dv_g(p) \\ 
			&= \int_\Sigma f\psi dv_g. 
		\end{aligned} \] 
		
		Finally, to prove uniqueness, let $w\in W^{1,q}(\Sigma)$ satisfy the corresponding homogeneous equation in the above weak sense. 
		For any $h\in C^\infty(\Sigma)$, let $\phi\in C^\infty(\Sigma)$ be the solution of
		\[\begin{cases}
			-\Delta_g\phi+\beta\phi=h &\text{in } \intsigma,\\
			\partial_{\nu_g}\phi=0    &\text{on }\partial\Sigma.
		\end{cases}    \]
		Such a solution exists since   $-\beta$ is not an eigenvalue   of $-\Delta_g$ restricted to the zero-mean Neumann space.
		Taking $\psi=\phi$ in the weak formulation gives
		$\int_\Sigma wh dv_g=0$.
		Since this holds for every $h\in C^\infty(\Sigma)$, we conclude that $w\equiv0$.
	\end{proof}
	
	\medskip
	
	We now begin the blow-up analysis. 
	Let $ \rho^n = (\rho^n_1, \rho^n_2) $ be a sequence of positive parameters that converges to $ \rho \in \mathbb{R}^2 $. Let $ u^n = (u^n_1, u^n_2) $ be  solutions of the Toda system~\eqref{eq:toda} corresponding to the parameter $ \rho^n$.
	

	For each $i=1,2$, let
	\begin{equation}\label{eq:def_v^n_i} 
		v^n_i = u^n_i -\log \int_{\Sigma} V_i e^{u^n_i} dv_g +\log \rho^n_i.
	\end{equation}
	Since $h_i>0$ and $V_ie^{v^n_i}$ is bounded in $L^1(\Sigma)$, up to a subsequence, there exists a non-negative finite measure $\mu_i\in \cM(\Sigma)$ such that $V_i e^{v_i^n}dv_g \rightharpoonup \mu_i$ as $n\to+\infty$ in the sense of measure, namely,
	\begin{equation}  \label{def:mu_i} 
		\int_{\Sigma}V_ie^{v^n_i}\phi \, dv_g=\int_{\Sigma} \frac{ \rho^n_i  V_i e^{u^n_i}}{\int_{\Sigma}  V_i e^{u^n_i} dv_g} \phi \, dv_g\rightarrow \int_{\Sigma} \phi \, d\mu_i, \qquad \text{for every } \phi \in C(\Sigma).
	\end{equation}
	For each $i=1,2$, define 
	\begin{equation*}
		\widetilde{\mathcal{S}}_i:=\left\{
		x\in\Sigma:\ \mu_i(x)\geq{\frac{\varrho(x)}{2}}\min\{ 1, 1+\gamma_i(x)\}\right\},
	\end{equation*}
	where $\mu(x_0):= \mu(\{x_0\})$. 
	\par

	Fix $i\in\{1,2\}$ and let $ x_0 \in\widetilde{\mathcal{S}}_i$. Then it holds that $\mu_i(x_0) \geq \frac{1}{2} \, \varrho(x_0) \min\{1, 1 + \gamma_i(x_0)\} > 0 $.
	Since $\mu_i(\Sigma)<\infty$, it follows that the set $\widetilde{\mathcal{S}}_i$ is finite for each $i=1,2$. This implies that $\mathcal{S}$  is finite by the following lemma.
	\par

	Firstly, we introduce a crucial estimate from \cite[Theorem~1]{brezis_uniform_1991} to prove the compactness-concentration property of \eqref{eq:toda}.
	\par

	\begin{lemma}[see~\cite{brezis_uniform_1991}]\label{Brezis_1991}
		Let $D \subseteq \mathbb{R}^2$ be a bounded domain, and let $f(x)\in L^1(D)$. Suppose that $u$ is a solution to the boundary value problem
		\begin{equation*}
			\begin{cases}
				\begin{aligned}
					-\Delta u&=f(x)\quad\text{ in } D,\\
					u&=0 \quad\text{ on } \partial D.
				\end{aligned}
			\end{cases}
		\end{equation*}
		\par
		
		Then for any $\delta \in(0,4 \pi)$, we have that
		\begin{equation*}
			\int_D \exp \left(\frac{(4 \pi-\delta)|u|}{\|f\|_{L^1(D)}}\right)\leqslant \frac{4 \pi^2}{\delta}(\operatorname{diam} D)^2.
		\end{equation*}
	\end{lemma}

	Recall that the blow-up point set of \eqref{eq:toda} with respect to $\{ u^n\}_n$ is given by
	\[\mathcal{S}:=\Big\{x\in \Sigma: \exists x_n\rightarrow x \text{ such that } \max_{i=1,2} v^n_i(x_n) \rightarrow +\infty \Big\}. \]
	\begin{lemma} \label{lem:S_i} 
		It holds that $\mathcal{S}=\bigcup^2\limits_{i=1}\widetilde{\mathcal{S}}_i$. Hence, the blow-up point set  $\mathcal{S}$ is a finite set. 
	\end{lemma}
	\begin{proof}
		{\it \textbf{(i)}.} Firstly, to prove $S\subset\bigcup^2\limits_{i=1}\widetilde{\mathcal{S}}_i$, fix $x_0\notin \bigcup^2\limits_{i=1}\widetilde{\mathcal{S}}_i$. 
		By the existence of interior and boundary isothermal coordinates, as described in \cite{chern1955,yang_wang2021,Esposito2014singular}, there exists a local chart $(y_{x_0}, U(x_0))$ around $x_0$ with $y_{x_0}(x_0)=0$ such that $y_{x_0}(U(x_0))=B^{x_0}$ where
		\[
		B^{x_0}=\begin{cases}
			\mathbb{B}_{r_0} & \text{ if } x_0 \in \intsigma,\\
			\mathbb{B}^+_{r_0} \cup I_{r_0} & \text{ if } x_0 \in \partial\Sigma
		\end{cases}
		\]
		for some sufficiently small radius $r_0>0$.
		In this coordinate, there exists a smooth function $\varphi_{x_0}:B^{x_0}\to \R$ such that $\varphi_{x_0}(0)=0$,  $g= e^{\varphi_{x_0}(y)}( dy_1^2 +dy_2^2)$ and $\Delta_g= e^{-\varphi_{x_0}(y)}\Delta_{\mathbb{R}^2}$.
		Furthermore, when $x_0 \in \partial\Sigma$, for any $x \in \partial\Sigma \cap U(x_0)$, $y_{x_0}(x) \in \{y \in \R^2 : y_2=0\}$ and 
		\[
		(y_{x_0})_*(\nu_g(x))=-e^{-\frac{1}{2}\varphi_{x_0}(y)} \frac{\partial}{\partial y_2}.
		\]
		\par

		Now we define $\tilde{u}^n_i$ on $\B_{r_0}$ as follows:
		\[
		\tilde{u}^n_i(y):= \begin{cases}
			(u^n_i\circ y^{-1}_{x_0})(y_1,y_2) & \quad \text{ if } x_0 \in \intsigma,\\
			(u^n_i\circ y^{-1}_{x_0})(y_1,|y_2|) & \quad \text{ if } x_0 \in \partial\Sigma.
		\end{cases}
		\]
		The functions $\tilde{v}_i^n$ and $\tilde{V}_i$ are defined in the same way by replacing $u_i^n$ with $v_i^n$ and $V_i$, respectively. We also define $\tilde{\varphi}_{x_0}$ by
		\[
		\tilde{\varphi}_{x_0}(y):= \begin{cases}
			\varphi(y_1,y_2) & \quad \text{ if } x_0 \in \intsigma,\\
			\varphi(y_1,|y_2|) & \quad \text{ if } x_0 \in \partial\Sigma.
		\end{cases}
		\]
		By the homogeneous Neumann boundary condition in the case $x_0\in\partial\Sigma$, the even extension across ${y_2=0}$ satisfies the reflected equation without producing any additional boundary term. Consequently, $\tilde{u}^n_i$ ($i=1,2$) satisfy the following local system
		\begin{equation*}
			-\Delta_{\mathbb{R}^2}\tilde{u}^n_i(y)=\tilde{f}^n_i(y),\quad y\in\mathbb{B}_{r_0}, 
		\end{equation*}
		in the distributional sense, where $\tilde{f}^n_i(y)$ is defined by 
		\[
		\tilde{f}^n_i(y):=-\beta e^{\tilde{\varphi}_{x_0}(y)} \tilde{u}^n_i(y)+ e^{\tilde{\varphi}_{x_0}(y)} \sum_{j=1}^2  a_{ij} \left(\tilde{V}_j(y)e^{\tilde{v}^n_j(y)} - \rho^n_j\right).
		\]
		For any real-valued function $f$, we denote its positive part by $f_+:=\max\{0,\ f\}$.

Fix $i \in \{1,2\}$. Since $-\Delta_g u^n_i+\beta u^n_i$ is uniformly bounded in $L^1(\Sigma)$ and  $-\beta$ is not an eigenvalue   of $-\Delta_g$ restricted to the zero-mean Neumann space, the standard potential estimate yields that, for any fixed $q \in (1,2)$,   
\begin{equation*}
	\|\nabla u^n_i\|_{L^q(\Si)}+ \|u^n_i\|_{L^q(\Si)}\leq\widetilde{C},
\end{equation*}
uniformly in $n$.

Since $x_0 \notin \widetilde{\mathcal{S}}_i$, we have $\mu_i(x_0)<{\frac {\varrho(x_0)}{2}}\min\{ 1,\,1+\gamma_i(x_0)\}$. 
Hence, there exist constants $r_1\in(0,\,r_0)$ and $c_0>0$ such that, for all sufficiently large $n$ and every $r\in (0,r_1)$,
\begin{equation*}
	\int_{\B_{r}}  \tilde{V}_ie^{\tilde{\varphi}_{x_0}+\tilde{v}^n_{i}}\leq c_0<2\pi\min\{1, 1+\gamma_i(x_0)\}\le2\pi.
\end{equation*}	
By the above $L^q$-estimate and H\"older inequality, there exists $c_1>0$ such that
\begin{equation*}
	\int_{\B_r}\left| -\beta e^{\varphi_{x_0}(y)}\tilde{u}^n_i(y) \right| dy\leq |\beta| \|u^n_i\|_{L^q(\Sigma)}\left( \int_{\B_r}  e^{\varphi_{x_0}(y)} dy\right)^{\frac {q-1} q}\leq c_1 r^{\frac{2(q-1)}{q}}, \quad r \in (0,r_1).
\end{equation*} 
Applying the above estimates to the definition of $\tilde{f}_i^n$ and taking $q=\frac3 2$, there exists a constant $c_2>0$ depending only on $(\Sigma,g)$ such that 
\begin{equation*}
	\int_{\B_r} |(\tilde{f}^n_i)_+|\leq 2c_0+ c_1 r^{\frac 2 3 } +c_2\sup_{n\in\mathbb{N},i\in\{1,2\}}|\rho_i^n|r^2, \quad \quad r \in (0,r_1).
\end{equation*}

We fix $A>1$ such that 
\[
\frac{1}{\min\{1,1+\gamma_i(x_0)\}}<A<\frac{2\pi}{c_0}
\]
and choose $\delta\in (0,4\pi)$ sufficiently small so that $1+\delta<A$ and $4\pi-\delta>2c_0(1+\delta)A$.
We define
\begin{equation*}
	r(\delta):=\min\left\{r_1,\sqrt{\frac {c_0\delta} {c_2\sup_{n\in\mathbb{N}, i\in\{1,2\}}|\rho_i^n|}}, \left(\frac{ c_0\delta}{c_1} \right)^{\frac 3 2 }\right\}.
\end{equation*}		   
Then we have $\|(\tilde{f}^n_i)_+\|_{L^1(\mathbb{B}(r(\delta)))} \le 2c_0(1+\delta)$.

We now decompose $\tilde{u}^n_i$ on the fixed ball $\mathbb{B}_{r_\delta}$. Let
\begin{equation*}
	\tilde{w}^n_i:=\tilde{u}^n_{i1}+\tilde{u}^n_{i2},
\end{equation*}
where $\tilde{u}^n_{i1}$ is the solution of
\begin{equation}\label{S_i:E-D-F}
	\left\{
	\begin{array}{ll}
		-\Delta_{\mathbb{R}^2}\tilde{u}^n_{i1}=(\tilde{f}^n_i)_{+}&\text { in }\mathbb{B}_{r(\delta)},\\
		\tilde{u}^n_{i1} =0& \text{ on }\partial\mathbb{B}_{r(\delta)},
	\end{array}\right.
\end{equation}
and $\tilde{u}^n_{i2}$ is the solution of
\begin{equation*}
	\begin{cases}
		-\Delta_{\mathbb{R}^2}\tilde{u}^n_{i2}=0\quad&\text{ in }\mathbb{B}_{r(\delta)},\\
		\tilde{u}^n_{i2}=(\tilde{u}^n_i)_+\quad&\text{ on }\partial\mathbb{B}_{r(\delta)}.	
	\end{cases}
\end{equation*}
By the comparison principle, we deduce that
\begin{equation*}
	\tilde{w}^n_{i}\geq\tilde{u}^n_i\text{ and }\tilde{u}^n_{i1}, \tilde{u}^n_{i2} \geq0\text{ in }\mathbb{B}_{r(\delta)}.
\end{equation*}
On the other hand, by the uniform $W^{1,q}$-estimate of $\tilde{u}_i^n$ and the trace inequality, we have $\|(\tilde{u}_i^n)_+\|_{L^1(\partial \B(r(\delta)))} \le C$.
Since $\tilde{u}^n_{i2}$ is harmonic in $\B_{r(\delta)}$ with boundary data $(\tilde{u}_i^n)_+$, the Poisson representation formula yields 
\begin{equation}\label{S_i-0}
	\|\tilde{u}_{i2}^n\|_{L^\infty(B_{r(\delta)/2})} \le C.
\end{equation}

We next estimate $\tilde{u}^n_{i1}$.         
Applying Lemma~\ref{Brezis_1991} to \eqref{S_i:E-D-F}, there exists some constant $C>0$ such that, 
\begin{equation*}
	\int_{\B_{r(\delta)}} e^{A\tilde{u}^n_{i1} } \le \int_{\B_{r(\delta)}} \exp \left(\frac{(4 \pi-\delta) \tilde{u}^n_{i1}}{\|(\tilde{f}^n_{i})_+\|_{L^1(\B_r)}}\right) \leq C. 
\end{equation*}
We now improve the integrability of $(\tilde{f}_i^n)_+$.
Since $\tilde{v}_i^n \le \tilde{u}_i^n+C \le \tilde{u}_{i1}^n+C$ in $\B_{r(\delta)/2}$,
by the local behavior of $\widetilde{V}_i$ and H\"older's inequality, there exists $p>1$ sufficiently close to 1 such that 
\[
\|(\tilde{f}^n_i)_+\|_{L^p(\B_{r(\delta)/2})} \le C.
\]
Here, $p>1$ can be chosen by the choice of $A$.
Therefore, by the standard elliptic estimate and Sobolev's inequality, we have
\begin{equation}\label{S_i-01}
	\|\tilde{u}^n_{i1}\|_{L^{\infty}(\mathbb{B}_{r(\delta)/4})}\leq C.
\end{equation}
\par

Combining \eqref{S_i-0} and \eqref{S_i-01}, we obtain
\begin{equation*}
	\tilde{u}^n_i(y)\leq \tilde{w}^n_i(y)= \tilde{u}^n_{i1}(y)+\tilde{u}^n_{i2}(y) \leq C,\quad\text{for all } y\in \B_{r(\delta)/4}.
\end{equation*}
Then, $u_i^n$ is uniformly bounded from above in a neighborhood of $x_0$.
Since we fixed $i \in \{1,2\}$ arbitrarily, both $u_1^n$ and $u_2^n$ are uniformly bounded from above near $x_0$.
Consequently, we have $x_0 \notin \mathcal{S}$, which implies that $\mathcal{S} \subset \bigcup^2\limits_{i=1}\widetilde{\mathcal{S}}_i$.

{\it \textbf{(ii)}.}  It suffices to prove that $\bigcup^2\limits_{i=1}\widetilde{\mathcal{S}}_i\subset\mathcal{S}$. For any $x_0\notin\mathcal{S}=\mathcal{S}_1\cup\mathcal{S}_2$,  there exists a sufficiently small neighborhood $U_{r_0}(x_0)$ and some constant $C>0$, such that
\begin{equation*}
	\max\limits_{U_{r_0}(x_0)} v^n_i \leq C.
\end{equation*}
Then, for small  $r\in (0, r_0)$, it follows that 
\begin{equation*}
	\begin{aligned}
		\int_{U_{r}(x_0)} V_i e^{v^n_i} dv_g\leq
		\begin{cases}
			C r^2,\ &\text{if}\ x_0\not\in\{p_{il}\}_{i=1,2;\,l=1,\cdots,\iota},\\
			Cr^{2(1+\gamma_{il})},\ &\text{if}\ x_0=p_{il},
		\end{cases} 
	\end{aligned}
\end{equation*}
for some constant $C>0$. This implies  $\mu_i(x_0)=0$, which means that  $x_0\notin\bigcup^2\limits_{i=1}{\widetilde{\mathcal{S}}}_i$. The lemma is concluded. 
\end{proof}
\par

By the convergence in the sense of measures stated in \eqref{def:mu_i}, together with the definition of the local limit masses $\sigma_i(x)$ in \eqref{def:masses}, we deduce that
\begin{equation}\label{W-C-r}
V_i e^{v_i^n}dv_g=\rho^n_i{\frac{V_ie^{u^n_i}}{\int_{\Sigma}V_ie^{u^n_i}dv_g}}dv_g\rightharpoonup\sum\limits_{x\in\mathcal{S}_i}\sigma_i(x)\delta_x+\mathrm{r}_idv_g\quad\text{ in sense of measures},
\end{equation}
where $\mathrm{r}_i\in L^1(\Sigma)\cap L^{\tilde{q}}(\Sigma)$ for some $\tilde{q}>1$ since $\min\limits_{1\leq i\leq2; 1\leq l\leq\iota}\{\gamma_{il}\}>-1$.
We also have $\mathrm{r}_i \ge 0$.
\par

\begin{lemma} \label{W-C-B}  
Let $G^g_x=G^g(\cdot,x)$ be the Green function given by \eqref{Green-F-G} and for $j=1,2$, $\sigma_j(x)$ be the local limit mass defined in \eqref{def:masses}. For any $i=1,2$, up to a subsequence,
\begin{equation*}
u^n_i=v^n_i-\overline{v}^n_i\to\sum_{j=1}^2 \sum_{x\in \cS_j} a_{ij}\sigma_j(x)G^g_x +s_i
\end{equation*} 
in $L^{\infty}_{loc}(\Sigma\setminus\mathcal{S})$ and weakly in $W^{1,q}(\Sigma)$ for any $q\in (1,2)$
where $s_i$ solves the following equations in the weak sense, with test functions in $C^\infty(\Sigma)$:
\begin{equation}\label{W-C-B-2} 
\begin{cases}
	-\Delta_g s_i+\beta s_i= \sum^2\limits_{j=1} a_{ij}( \mathrm{r}_j+ \sum\limits_{x\in\cS_j} \sigma_j(x) - \rho_j) & \text{ in } \intsigma,\\
	\partial_{\nu_g} s_i=0 & \text{ on } \partial \Sigma,\\
	\int_{\Sigma} s_i dv_g =0,
\end{cases}
\end{equation}
with $\mathrm{r}_i$ defined in \eqref{W-C-r} and $e^{s_i}\in L^p(\Sigma)$ for any $p \ge 1$.
\end{lemma}
\begin{proof}
Since $-\Delta_gu^n_i+\beta u^n_i$ is uniformly bounded in $L^1(\Sigma)$, Lemma \ref{cor:Green-potential} yields that, for any $q \in (1,2)$,
\begin{equation}\label{W-C-B-4}
\|\nabla_gu^n_i\|_{L^q(\Sigma)}+\|u^n_i\|_{L^q(\Sigma)}\leq C
\end{equation}
uniformly in $n$.
Thus, after taking a subsequence if necessary, $u^n_i$ converges to some $u_i$ weakly in $W^{1,q}(\Sigma)$ for any $q \in (1,2)$.
Moreover, since $V_ie^{v_i^n}$ is uniformly bounded in $L^\infty_{loc}(\Sigma \setminus \mathcal{S})$, 
using the standard elliptic estimates, we obtain that $u_i^n$ is uniformly bounded in $W^{2,q}_{loc}(\Sigma \setminus \mathcal{S})$, and this implies that $u_i^n \to u_i$ in $L^\infty_{loc}(\Sigma \setminus \mathcal{S})$.

\medskip

Now define $s_i:=u_i-\sum^2_{j=1}a_{ij}\sum\limits_{x\in\mathcal{S}_j}\sigma_j(x)G^g_x$. Combining \eqref{eq:toda} and \eqref{W-C-r}, $s_i$ satisfies the following equation, in the weak sense:
\begin{equation*}
\begin{aligned}
	\begin{cases}
		-\Delta_g s_i=-\beta s_i+\sum^2_{j=1} a_{ij}\big( \mathrm{r}_j+ \sum\limits_{x\in\cS_j} \sigma_j(x) - \rho_j\big)\quad&\text{ in }\intsigma,\\
		\partial_{\nu_g}s_i=0\quad&\text{on}\ \partial\Sigma,\\
		\int_{\Sigma}s_i dv_g=0,
	\end{cases}
\end{aligned}
\end{equation*} 
where $s_i\in W^{1,q}(\Sigma)$.		
Moreover, since $-\Delta_g s_i \in L^1(\Sigma)$ and $s_i \in W^{1,q}(\Sigma)$ for any $q \in (1,2)$, applying Remark 2 in \cite{brezis_uniform_1991}, we obtain $e^{s_i} \in L^p(\Sigma)$ for any $p \ge 1$.	 
\end{proof}
\par




\begin{lemma}\label{ri}	
For any $i=1,2$, $\mathrm{r}_i\equiv0$ holds if and only if $\overline{v}^n_i=-\log \int_\Sigma V_i e^{u^n_i} dv_g +\log \rho_i^n\rightarrow -\infty$ as $n\to+\infty$. Moreover, if $\mathrm{r}_i\equiv0$, then $\rho_i=\sum\limits_{x\in\mathcal{S}_i}\sigma_i(x)$.
\end{lemma}
\begin{proof}
First, the equivalence of these two conditions follows from the same argument as in the case $\partial\Sigma=\emptyset$; see \cite[Lemma 2.2]{Battaglia2015a_note}.
Moreover, by \eqref{W-C-r}, we have $\rho_i=\lim_{n\to\infty} \rho_i^n=\sum\limits_{x\in\mathcal{S}_i}\sigma_i(x)$ when $\mathrm{r}_i\equiv0$.
\end{proof}
\par

\begin{remark}\label{Bdd-formula-ri}


By Lemma \ref{W-C-B} and Lemma \ref{ri}, we have $\mathrm{r}_i=\hat{V}_ie^{s_i}$, where
\[
	\hat{V}_i:=V_i\exp\Big(\lim\limits_{n\to\infty}\overline{v}^n_i+\sum^2_{j=1}\sum_{x\in\mathcal{S}_j}a_{ij}\sigma_j(x)G^g_x \Big).
\]
Moreover, if $\mathrm{r}_i\not\equiv0$, then $\hat{V}_i\simeq d_g(\cdot,x)^{2\gamma_i(x)-{\frac{\sum^2_{j=1}2a_{ij}\sigma_j(x)}{\varrho(x)}}}$ around $x\in\mathcal{S}_i$.
\end{remark}
\par

We now present a concentration-compactness alternative theorem, originally established in \cite{Lucia2002Vortex} for domains in $ \R^2 $, and extended in \cite{battaglia2013moser} to the setting of closed surfaces.
\begin{theorem}\label{thm:compactness_concentration} One of the following statements alternatively occurs: 
\begin{itemize}
	\item [a.] (Compactness) $\{v^n\}$ is uniformly bounded in $H^1(\Sigma)\times H^1(\Sigma)$ and $L^\infty (\Sigma)\times L^\infty(\Sigma)$. 
	\item [b.](Concentration) The blow-up set $\cS=\cS_1\cup \cS_2$ is nonempty and finite. For each $i=1,2$, either $v^n_i$ is bounded in $L_{loc}^{\infty}(\Sigma\setminus \cS_i)$ or $v^n_i\rightarrow -\infty$ locally uniformly in $\Sigma\setminus \cS_i.$ \\
	In addition, if $\cS_i \setminus (\cS_1\cap \cS_2) \neq \emptyset$ for some $i=1,2$, then $v^n_i \rightarrow -\infty$ locally uniformly in $\Sigma\setminus \cS$. 
\end{itemize} 
Moreover, 
$
\mu_i =\mathrm{r}_i dv_g + \sum_{x \in S_i} \sigma_i(x) \delta_x
$
with $\mathrm{r}_i$ defined in \eqref{W-C-r} and $\sigma_i(x) \geq \frac{\varrho(x)}{2} \min\{1,1+\gamma_i(x)\}$ for any $ x \in \cS_i, i = 1, 2$.
\end{theorem}
\begin{proof}
{\bf Case I.} When $\mathcal{S}=\mathcal{S}_1\cup\mathcal{S}_2=\emptyset$.\\
Then $v^n_i$ is uniformly bounded from above for $i=1,2$. Moreover, 
\begin{equation*}
	\big|-\Delta_g u^n_i+\beta u^n_i\big|\leq\sum^2\limits_{j=1}|a_{ij}|\big(V_je^{v^n_j}+\rho^n_j\big)\leq C\big(\sum^2\limits_{j=1}V_j+1\big),
\end{equation*}
for some constant $C>0$. This implies that $-\Delta_gu^n_i+\beta u^n_i$ is uniformly bounded in $L^{\tilde{q}}(\Sigma)$ for some $\tilde{q}>1$ since $\min\limits_{1\leq i\leq2;\,1\leq l\leq\iota}\{\gamma_{il}\}>-1$.

By the standard elliptic estimate, $u^n_i$ is uniformly bounded in $\oH$ and $L^\infty$. This implies the uniform boundedness of $\int_{\Sigma}V_ie^{u^n_i}$. Thus, $v^n_i=u^n_i-\log\int_{\Sigma}V_ie^{u^n_i}+\log\rho^n_i$ is uniformly bounded in $H^1$ and $L^\infty$. 
Hence, the alternative (a) occurs.
\medskip

\noindent {\bf Case II.} When $\mathcal{S}=\mathcal{S}_1\cup\mathcal{S}_2$ is nonempty.\\
In this case, $\mathcal{S}$ is finite by Lemma \ref{lem:S_i}. 
Moreover, by Lemmas \ref{W-C-B} and \ref{ri}, $v^n_i\to-\infty$ in any compact subset of $\Sigma\setminus\mathcal{S}_i$ if $\mathrm{r}_i\equiv0$; and $v^n_i$ is bounded in $L^{\infty}_{loc}(\Sigma\setminus\mathcal{S}_i)$ if $\mathrm{r}_i\not\equiv0$.
Hence, the alternative (b) occurs.

It remains to prove the additional assertion in Case II. 
Suppose that $x_0\in \cS_i\setminus(\cS_1\cap\cS_2)$ and let $j\neq i$. 
To the contrary, assume that $\mathrm r_i\not\equiv0$. 
By Lemma \ref{ri}, $\overline v_i^n$ remains bounded, and hence Lemma \ref{W-C-B} implies that $\{v_i^n\}_{n\in\N}$ is locally uniformly bounded in $\Sigma\setminus\cS$. 
Since $x_0\notin\cS_j$, the nonlinear term corresponding to the $j$-th component is locally controlled near $x_0$. 
Moreover, the term involving $\beta u_i^n$ is bounded in $L^p$ for some $p>1$ by \eqref{W-C-B-4}. 
Therefore, after absorbing the term involving the $j$-th component and the term $\beta u_i^n$ by solving the corresponding local Poisson problem, the equation for the $i$-th component reduces to a scalar Liouville-type equation. 
Then, the argument of \cite[Theorem 3]{brezis_uniform_1991} rules out the possibility that the $i$-th component blows up at $x_0$ while $\{v_i^n\}_n$ remains locally uniformly bounded in $\Sigma\setminus\cS$.
This contradiction proves that $\mathrm{r}_i\equiv 0$. 
The conclusion now follows immediately from Lemmas \ref{W-C-B} and \ref{ri}.

\medskip

Finally, by \eqref{W-C-r} and the definition of $\widetilde{\mathcal{S}}_i$, we have $\mu_i=\mathrm{r}_idv_g+\sum\limits_{x\in\mathcal{S}_i}\sigma_i(x)\delta_x$ and $\sigma_i(x)\geq{\frac{\varrho(x)}{2}}\min\{1,1+\gamma_i(x)\}$ for any $x\in\mathcal{S}_i$. 
\end{proof}

The following lemma can be derived from Lemma~\ref{lem:S_i} and Lemma~\ref{W-C-B} using standard arguments; see, for instance, \cite{Battaglia2015a_note}. For completeness, we provide the proof.
\begin{lemma}\label{lemma:w2p}
For $i=1,2$, $s_i \in W^{2,p}(\Sigma)$ for some small $p>1$. 
Moreover, if $\mathrm{r}_{i}\not\equiv0$ for some $i\in\{1,\,2\}$, then $\sum_{j=1}^2 a_{ij} \sigma_j(x)<\varrho(x)\big(1+\gamma_{i}(x)\big)$ for all $x\in\mathcal{S}_{i}$.
\end{lemma}
\begin{proof} 
For convenience, denote \[ \Lambda_i(x):=\sum_{j=1}^2a_{ij}\sigma_j(x), \qquad i=1,2. \] 
Recall that $s_i$ satisfies \eqref{W-C-B-2} and $\mathrm{r}_i=\hat{V}_i e^{s_i}$.
If $\mathrm{r}_1\equiv \mathrm{r}_2\equiv0$, then Lemma \ref{ri} gives 
$\rho_i=\sum_{x\in\mathcal S_i}\sigma_i(x)$ for $i=1,2$.
Hence, \[ -\Delta_gs_i+\beta s_i=0 \quad\text{in }\intsigma, \qquad \partial_{\nu_g}s_i=0 \quad\text{on }\partial\Sigma. \] 
Since  $-\beta$ is not an eigenvalue   of $-\Delta_g$ restricted to the zero-mean Neumann space, we obtain $s_i\equiv0$ for $i=1,2$.

We now assume that at least one residual is nontrivial. By Lemma \ref{W-C-B} and \eqref{W-C-r}, 
\[ 
s_i\in L^\infty_{loc}(\Sigma\setminus\mathcal S), \qquad \mathrm r_i\in L^{\tilde{q}}_{loc}(\Sigma\setminus\mathcal S) \quad \text{ for some } 1<\tilde{q}<\infty.
\] 
Moreover, by Remark \ref{Bdd-formula-ri}, near any $x_0\in\mathcal S_i$, 
\begin{equation}\label{eq:local-weight-w2p} 
\hat{V}_i(x) \simeq d_g(x,x_0)^{ 2\gamma_i(x_0) -\frac{2}{\varrho(x_0)}\Lambda_i(x_0) }. 
\end{equation}

For later use, set
\[
c_j:=\sum_{x\in\mathcal S_j}\sigma_j(x)-\rho_j,
\qquad j=1,2.
\]
The right-hand side of \eqref{W-C-B-2} has zero mean, by the homogeneous Neumann condition and the normalization
$\int_\Sigma s_i\,dv_g=0$. 
Hence, by Lemma \ref{cor:Green-potential}, we have the representation formula
\begin{equation*}
s_i(x)=\sum_{j=1}^2a_{ij}\int_\Sigma G^g(x,p)\mathrm r_j(p) dv_g(p)+H_i(x),
\end{equation*}
where
\[
H_i(x):=\sum_{j=1}^2a_{ij}c_j \int_\Sigma G^g(x,p) dv_g(p).
\]
By the logarithmic estimate for $G^g$ in Lemma \ref{lem:Green-properties},
it holds that $\sup_{x\in\Sigma} \int_\Sigma |G^g(x,p)| dv_g(p)<\infty$.
Thus, $H_i\in L^\infty(\Sigma)$.

\medskip

\noindent\textbf{Case I.} Suppose that exactly one residual is nontrivial. Without loss of generality, let $\mathrm r_1\not\equiv0, \mathrm r_2\equiv0.$
By the Green representation associated with \eqref{W-C-B-2}, 
\[ 
s_1(x) = 2\int_\Sigma G^g(x,y)\mathrm{r}_1(y) dv_g(y)+H_1(x), 
\] 
where $H_1$ is bounded on $\Sigma$. 
Since $G^g(x,y)\geq-C$ and $\mathrm{r}_1\in L^1(\Sigma)$, it follows that $s_1\geq -C$.
Therefore, for every $x_0\in\mathcal S_1$, \eqref{eq:local-weight-w2p} gives 
\[ \mathrm r_1(x) \geq C\,d_g(x,x_0)^{ 2\gamma_1(x_0) -\frac{2}{\varrho(x_0)}\Lambda_1(x_0) } \] 
near $x_0$. 
Since $\mathrm r_1\in L^1(\Sigma)$, necessarily 
\[ \Lambda_1(x_0) < \varrho(x_0) \bigl(1+\gamma_1(x_0)\bigr). \] 
Hence, the relation $\mathrm{r}_i=\hat{V}_i e^{s_i}$, together with the exponential integrability of $s_1$ in Lemma \ref{W-C-B}, implies by H\"older's inequality that $\mathrm{r}_1\in L^p(\Sigma)$ for some $p>1$.
Since $\mathrm{r}_2\equiv0$, equation \eqref{W-C-B-2} and the global Neumann $W^{2,p}$-estimate yield 
$s_i\in W^{2,p}(\Sigma)$, $i=1,2$.

\medskip

\noindent\textbf{Case II.} Suppose that both residuals $\mathrm{r}_1, \mathrm{r}_2$ are nontrivial. 
By Theorem \ref{thm:compactness_concentration}, we have $\mathcal S_1=\mathcal S_2=\mathcal S$.
Set \[ w:=\frac{2s_1+s_2}{3}. \] 
Then equation \eqref{W-C-B-2} and the Green representation give 
\[ w(x) = \int_\Sigma G^g(x,y)\mathrm{r}_1(y)\,dv_g(y)+H(x), \] 
where $H=\frac{2H_1+H_2}{3} \in L^\infty(\Sigma)$. 
Thus, it holds that $w\geq-C$.

Using the convexity of the exponential function, we obtain 
\[ 
C\int_\Sigma\min\{\hat{V}_1,\hat{V}_2\}\,dv_g \leq \int_\Sigma \min\{\hat{V}_1,\hat{V}_2\}e^w\,dv_g 
\leq \frac23\int_\Sigma \hat{V}_1e^{s_1}\,dv_g + \frac13\int_\Sigma \hat{V}_2e^{s_2}\,dv_g  <\infty. 
\] 
Consequently, for every $x_0\in\mathcal S$, there exists $i\in\{1,2\}$ such that 
\[ \Lambda_i(x_0) < \varrho(x_0) \bigl(1+\gamma_i(x_0)\bigr). \] 
Fix $x_0\in\mathcal S$ and suppose, without loss of generality, that the above inequality holds for $i=1$.
As in Case I, there exist $p>1$ and $r>0$ such that $\mathrm{r}_1\in L^p\bigl(U_r^g(x_0)\bigr)$.

For $x\in U_{r/2}^g(x_0)$, the Green representation for $s_2$ can be written as 
\[ \begin{aligned} 
s_2(x) =& 2\int_\Sigma G^g(x,y)\mathrm{r}_2(y) dv_g(y) -\int_{U_r^g(x_0)}G^g(x,y)\mathrm{r}_1(y) dv_g(y) \\ 
&\quad -\int_{\Sigma\setminus U_r^g(x_0)} G^g(x,y)\mathrm{r}_1(y) dv_g(y) +H_2(x), 
\end{aligned} \] 
where $H_2$ is bounded. 
By $L^{p'}$-bound for the Green function and its boundedness away from the diagonal, using H\"older's inequality, we deduce that $s_2\geq-C$ in $U_{r/2}^g(x_0)$.
Applying \eqref{eq:local-weight-w2p} to $\mathrm{r}_2$ and using $\mathrm{r}_2\in L^1(\Sigma)$, we obtain 
\[ \Lambda_2(x_0) < \varrho(x_0) \bigl(1+\gamma_2(x_0)\bigr). \] 
The same argument then yields $\mathrm{r}_2\in L^{p_2}\bigl(U_{r/2}^g(x_0)\bigr)$ for some $p_2>1$. 
Since $x_0$ is arbitrary and $\mathcal S$ is finite, we may choose $p>1$ such that $\mathrm{r}_i\in L^p(\Sigma)$, $i=1,2$.
Finally, \eqref{W-C-B-2} and the global $W^{2,p}$-estimate for the homogeneous Neumann problem imply 
$s_i\in W^{2,p}(\Sigma)$, $i=1,2$.
This also proves the strict inequality for nontrivial residual $\mathrm{r}_i \not\equiv 0$. 
\end{proof}

We now derive algebraic constraints on the local limiting masses. 
The following Pohozaev-type identity is obtained by adapting the local arguments in \cite{jostandwang2001,Lucia2002Vortex} to the present setting.


\begin{lemma}\label{lemma:identity}
If $\mathcal{S}\neq\emptyset$, then for any $x\in\mathcal{S}$, $\sigma_1(x)$ and $\sigma_2(x)$ satisfy the following identity:
\begin{equation*}
\sigma_1^2(x)+\sigma_2^2(x)-\sigma_1(x)\sigma_2(x)=\varrho(x)\big(1+\gamma_1(x)\big)\sigma_1(x)+\varrho(x)\big(1+\gamma_2(x)\big)\sigma_2(x).
\end{equation*}
\end{lemma}
\begin{proof}
Fix $x_0\in\mathcal S$. 
If $x_0\in\intsigma$, we work in isothermal coordinates centered at $x_0$. 
If $x_0\in\partial\Sigma$, we use boundary isothermal coordinates and extend the transformed system evenly across the flat boundary, as in the proof of Lemma \ref{lem:S_i}. 
By Lemma~\ref{lemma:w2p}, the regular parts belong to $W^{2,p}$ for some $p>1$. 
Hence, the regular and lower-order terms make no contribution in the local Pohozaev identity as the radius tends to zero. 
Applying the standard local Pohozaev identity to the transformed system and taking first $n\to\infty$ and then the radius to zero, we obtain the stated identity. 
In the boundary case, the reflection doubles the local masses, and the resulting factor is incorporated in $\varrho(x_0)$.
\end{proof}
\par

{ The local mass classification for singular $SU(3)$ Toda systems in planar domains was proved in \cite[Theorem 1.2]{LinWeiZhang2015}. 
For interior blow-up points, the same argument can be used in local isothermal coordinates, since the lower-order term involving $\beta$ does not affect the local blow-up analysis. 
For boundary blow-up points, we use boundary isothermal coordinates and the even reflection as in Lemma \ref{lem:S_i}. 
The reflected coefficients may not be smooth across the flat boundary, but they have enough regularity on each side. 
The calculations involving their derivatives are done separately on the two half-disks, and the boundary terms vanish or cancel. 
Hence, the argument in \cite{LinWeiZhang2015} also applies to the boundary case.
}

\begin{proposition}\label{lemma:classification}
Assume that $\mathcal{S}\neq\emptyset$. Then, for every $x\in \mathcal{S}$,
\[
\bigl(\sigma_1(x),\sigma_2(x)\bigr) \in \Lambda_{\gamma_1(x),\gamma_2(x)}(x).
\]
\end{proposition}

We next show that, if blow-up occurs, there exists $i\in\{1,2\}$ such that $r_i\equiv 0$. Together with Proposition \ref{lemma:classification}, this implies that the corresponding parameter $\rho_i$ belongs to $\Gamma_i$.

\begin{lemma}\label{lemma:gre}
If $\mathcal{S}\neq\emptyset$, then for any $x\in\mathcal{S}$, 
there exists $i \in \{1,2\}$, such that 
\begin{equation*}
\sum\limits_{j=1}^2 a_{ij}\sigma_j(x)\geq \varrho(x)\big(1+\gamma_i(x)\big).
\end{equation*}
Hence, there exists $i \in \{1,2\}$, such that $\mathrm{r}_i\equiv 0$ and $\rho_i= \sum\limits_{x\in\mathcal{S}_i}\sigma_i(x)$.
\end{lemma}
\begin{proof}
To the contrary, suppose that there exists some $x_0\in\mathcal{S}$, such that 
$ 2\sigma_1(x_0)-\sigma_2(x_0)< \varrho(x_0)(1+\gamma_1(x_0))$ and $  2\sigma_2(x_0)-\sigma_1(x_0)< \varrho(x_0)(1+\gamma_2(x_0)).$
Applying the estimates above and $(\sigma_1(x_0),\sigma_2(x_0))\neq (0,0)$, it follows 
\begin{equation*}
\sigma_1^2(x_0)+\sigma^2_2(x_0)-\sigma_1(x_0)\sigma_2(x_0)< \frac {\varrho(x_0)\big(1+\gamma_1(x_0)\big)}{2}\sigma_1(x_0)+\frac{\varrho(x_0)\big(1+\gamma_2(x_0)\big)}{2} \sigma_2(x_0).
\end{equation*}
This result contradicts Lemma~\ref{lemma:identity}.

Moreover, if $\mathrm{r}_i \not\equiv 0$ for both $i=1,2$, then by Lemma \ref{lemma:w2p}, it holds that
\begin{equation*}
\sum\limits_{j=1}^2 a_{ij}\sigma_j(x)< \varrho(x)\big(1+\gamma_i(x)\big), \quad i=1,2.
\end{equation*}
This is a contradiction, so $\mathrm{r}_i \equiv 0$ for some $i \in \{1,2\}$.
\end{proof}
\par


Now, we are ready to prove the compactness theorem. 
\begin{proof}[Proof of Theorem~\ref{thm:cpt_sols}] Let $ \rho^n = (\rho^n_1, \rho^n_2) \rightarrow \rho \in\mathcal{A} $ as $n\rightarrow +\infty$ and $ u^n = (u^n_1, u^n_2) $ be the solutions of the Toda system~\eqref{eq:toda} corresponding to the parameter $ \rho^n $.

If $v^n_1$ and $v^n_2$, defined by \eqref{eq:def_v^n_i}, both have an upper bound, then by Theorem~\ref{thm:compactness_concentration}, $\{v^n\}$ is uniformly bounded in $H^1(\Sigma)\times H^1(\Sigma)$ and $L^\infty(\Sigma)\times L^\infty(\Sigma)$.
Hence, by the $W^{2,p}$-estimate for some $p>1$, we can deduce that $u^n$ is uniformly bounded in $W^{2,p}(\Sigma)$. Thus, $u^n$ is compact in $\cH,$ through the Rellich–Kondrachov compact embedding theorem. \par

If $\cS\neq \emptyset$, Lemma~\ref{lemma:gre} yields that there exists $i\in \{1,2\}$ such that $\rho_i=\sum_{x\in \cS_i}\sigma_i(x).$ We may assume $i=1.$ Then $\rho_1\in \Gamma_1$, which contradicts the assumption $\rho \in \cA \subset \R^2_{>0} \setminus \Gamma$. 
Hence, Theorem~\ref{thm:cpt_sols} is concluded. 
\end{proof}

\section{Analysis of Sublevels}\label{sec:3}
In this section, we always assume that $\cQ:=\{ p_l:l=1,\cdots, \iota\}\subset \intsigma$
and the singular weight $\gamma_{il}\geq 0$ for $i=1,2,l=1,\cdots,\iota$.

As explained in the introduction, to avoid component interaction and singular points, we will consider two retractions of $\Sigma$ onto two disjoint, simple closed curves  $\bga_1,\bga_2\subset \pa \Sigma$, which are non-contractible in $\Sigma$; we denote them by $\cP_i: \Sigma\rightarrow \bga_i$ for $i=1,2$. As will be clear later on, the argument works fine since the weight of the points on the boundary is smaller than that of interior points.

Recall that $-\beta$ is not an eigenvalue of $-\Delta_g$ restricted to the zero-mean Neumann spaces. Throughout this section, let $K_1,K_2, I \in \N_0$ be defined as in the introduction, so that $\rho_i \in (2\pi K_i, 2\pi  (K_i+1))$, $i=1,2$ and $-\lambda_{I+1}<\beta <-\lambda_{I}$.
Let $\{\psi_i\}_{i\in \mathbb{N}_+}$ be an orthonormal basis of $\oH$ consisting of Neumann eigenfunctions such that $-\Delta_g \psi_i=\lambda_i\psi_i$.
We also define projections to divide $u_i$ into positive and negative parts of the operator $-\Delta_g+\beta$. 
Specifically, for any $u\in\oH$, we define $\Pi_{I}: \oH\rightarrow \mathbb{R}^{I}$ by
\[ \Pi_{I}(u)= (\lan u, \psi_1\ran_{\oH},\cdots,\lan u,\psi_{I}\ran_{\oH}),  \]
and set $\Pi_0\equiv 0.$


For $k,k'\in \N_0$, we define 
\begin{equation*}
\begin{split}
\Sigma_{k,k'}:=\Big\{ \sum_{j=1}^{k} t_j\delta_{x_j}+\sum_{j=k+1}^{k+k'}t_j \delta_{x_j}:  
\sum_{j=1}^{k+k'}  t_j= 1, t_{j}\geq 0, & x_j\in \intsigma \text{ for } j=1,\dots, k, \\
& x_{j}\in \partial\Sigma \text{ for } j=k+1,\dots, k+k' \Big\},
\end{split}
\end{equation*}
where $\delta_x$ is the Dirac mass of $\Sigma$ concentrated at $x\in \Sigma$.  Then, the barycenter space of $\Sigma
$  is given by 
\[ (\Sigma_{\partial})_K:= \cup_{k=0}^{\lfloor K/2\rfloor} \Sigma_{k, K-2k}, \]
where $\lfloor  a\rfloor$ denotes the greatest integer less than or equal to $a$,for any $a\in \R. $
Here, we  treat the points in $\intsigma$ and on $\partial\Sigma$ differently, as the blow-up points on the boundary exhibit only half the singularity compared to those in the interior.

On signed measure space $\cM(\Sigma)$, we use the distance induced by the duality of Lipschitz functions, i.e. for any $\mu,\mu^{\prime}\in\cM(\Sigma) $
\begin{eqnarray*}
\|\mu\|_{Lip^\prime(\Sigma)}:=\sup_{h\in Lip(\Sigma), \|h\|_{Lip(\Sigma)}=1}\left| \int_{\Sigma} h d\mu\right| \text{ and } \dd(\mu,\mu'):= \|\mu-\mu'\|_{Lip^\prime(\Sigma)}.
\end{eqnarray*}
Since  $(\Sigma_\partial)_K\subset \cM(\Sigma)$ is a Euclidean Neighborhood Retract (ENR) as in \cite[Lemma 3.1]{Battaglia2018} and \cite[Lemma 4.1]{Ndiaye2008}, we have the following retraction lemma:
\begin{lemma}\label{lem:retraction}
Given $ K \in \mathbb{N} $, for sufficiently small $ \varepsilon_0 >0$, there exists a continuous retraction
\[ \phi_K : \{ \sigma \in \cM(\Sigma), \dd(\sigma, (\Sigma_{\partial})_K )< \varepsilon_0 \} \rightarrow (\Sigma_{\partial})_K . \]
In particular, if $ \sigma_n \rightarrow \sigma $ in the sense of measures with $ \sigma \in (\Sigma_{\partial})_K  $, then $ \phi_K(\sigma_n) \rightarrow \sigma $.
\end{lemma}
We consider the following barycenter space of  $\bga_i\subset \pa \Sigma$, $i=1,2$: For $K_i \ge 1$,
\[ (\bga_i)_{\partial,K_i}:=\Big\{\sum_{j=1}^{K_i} t^i_j\delta_{x^i_j}:  \sum_{j=1}^{K_i} t^i_j=1,  t^i_j \geq 0, x^i_j\in \bga_i\Big\}.\]
To deal with negative eigendirections, we join $(\bga_i)_{\partial, K_i}$ with the normalized sphere $\mathbb{S}^{I-1}$. When $K_i, I \ge 1$, we set
$(\bga_i)_{\partial,K_i,I}:= \mathbb{S}^{I-1} \star (\bga_i)_{\partial,K_i}.$
If $K_1, K_2$ or $I=0$, we set
\[  (\bga_i)_{\pa, K_i, I}:=\begin{cases}
(\bga_i)_{\pa, K_i}     & I=0, K_i\neq 0\\
\mathbb{S}^{I-1} & I\neq 0, K_i=0\\
\emptyset & I=K_i=0
\end{cases}.\]
Here and below, we use the convention for empty factors of the join stated in the Introduction.

We will construct a map from  $(\bga_1)_{\partial,K_1,I}\star (\bga_2)_{\partial,K_2,I}:=\mathbb{S}^{I-1}\star (\bga_1)_{\partial,K_1}\star \mathbb{S}^{I-1} \star  (\bga_2)_{\partial,K_2}$ into some low energy sublevel set $J_{\beta,\rho}^{-L}$. 
Let  $$\zeta:=(X_1, \sigma_1, X_2, \sigma_2;s)\in (\bga_1)_{\partial,K_1,I}\star (\bga_2)_{\partial,K_2,I},$$ where $\sigma_i=\sum_{j=1}^{K_i} t^i_j \delta_{x^i_j}$ in $(\bga_i)_{\partial,K_i}$ with $  x^i_j\in \bga_i$ and $\sum_{j=1}^{K_i} t^i_j=1$, $X_i=(X_{i,1},\cdots, X_{i,I})$ in $\mathbb{S}^{I-1}$ for any $i=1,2$ and $s=(s_1,s_2,s_3,s_4) \in \Delta^3$.

    Let $\Lambda>0$ be a large parameter. 
    If $K_i=0$, we set $v_i \equiv 0$, and if $I=0$, we set $w_i \equiv 0$ for $i=1,2$. 
	Otherwise, we define functions $v_i=v_i^\Lambda(\zeta), w_i=w_i^\Lambda(\zeta) \in H^1(\Sigma)$ as follows: for $i=1,2$,
	\begin{equation}\label{testftns}
		\begin{split}
			v_1:=\log\left(  \sum_{j=1}^{K_1}  \frac{t^1_j  }{(1+  s_2^2 \Lambda^2  d_g^{\,2}(\cdot,x_j^1))^2}   \right), \quad &	v_2:=\log\left(  \sum_{j=1}^{K_2}  \frac{t^2_j}{(1+  s_4^2\Lambda^2  d_g^{\,2}(\cdot,x_j^2))^2}   \right),\\
			w_1:= \sqrt{\log^+(s_1\Lambda)}\sum_{j=1}^{I} X_{1,j}\psi_j, \quad & w_2:= \sqrt{\log^+(s_3\Lambda)}\sum_{j=1}^{I} X_{2,j}\psi_j,
		\end{split}
	\end{equation} 
	where  $\log ^+(\lambda):=\log \max\{1,\lambda\}$ for any $\lambda \ge 0$. In the next result, we crucially exploit the fact that $x^i_j \in \gamma_i$ and the curves are disjoint and that $\gamma_i \subset \partial\Sigma$. We can also neglect the effect of
the singularities since on the boundary we have just regular points.
	
	\begin{lemma}\label{lemma:map}
		Let $K_1, K_2, I \in \N_0$ with $K_1+K_2+I>0$.  Assume that   $(\rho_1,\rho_2)\in  (2\pi K_1, 2\pi(K_1+1))\times  (2\pi K_2, 2\pi(K_2+1)) \setminus\Gamma$ and $\beta\in (-\lambda_{I+1},-\lambda_{I})$.
		Define functions $\varphi_i=\varphi_i^\Lambda(\zeta) \in \oH$, $i=1,2$ by
		\begin{equation}\label{def:zeta}
			\left(\begin{array}{c}
				\varphi_1^\Lambda\\
				\varphi_2^\Lambda
			\end{array}\right)=\left(\begin{array}{c}
				(v_1-\overline{v_1} + w_1)-\frac 1 2 (v_2-\overline{v_2}+w_2 )\\
				(v_2-\overline{v_2}+w_2)-\frac 1 2 (v_1-\overline{v_1} + w_1)
			\end{array}\right).
		\end{equation}
		Then, $ J_{\beta,\rho}(\varphi_1^\Lambda,\varphi_2^\Lambda)\rightarrow -\infty, \text{ as } 	 \Lambda\rightarrow +\infty$
		uniformly for $\zeta\in(\bga_1)_{\partial,K_1,I}\star (\bga_2)_{\partial, K_2,I}.$
	\end{lemma}
	\begin{proof}
		Let $\lambda_{11}=s_1\Lambda, \lambda_{21}=s_2\Lambda,\lambda_{12}= s_3\Lambda,\lambda_{22}= s_4\Lambda$. 
		A standard computation, see \cite{battaglia_jevnikar_ruiz2015}, yields 
		\begin{eqnarray}
			\left|\int_{\Sigma} \lan \nabla v_1,\nabla v_2\ran_g dv_g \right|&\leq& C,\label{eq: est_mix_v12}
					\end{eqnarray}
			since		$x^i_j \in \gamma_i$ and the curves are disjoint. Moreover, 
					\begin{eqnarray}
			\int_{\Sigma} |\nabla v_i|_g^2 dv_g&\leq& \sum_{j=1}^{K_i} 16 \pi (1+o(1))\log^+(\lambda_{2i}) + C \label{eq:est_int_nabla_vi}
		\end{eqnarray}
		where  we  observe that one has half of the Dirichlet energy contribution due to the fact that  $x^i_j \in \gamma_i \subset \partial\Sigma$. 
		We claim that
						\begin{eqnarray}
			\left|	\int_{\Sigma} \lan \nabla v_i,\nabla w_{i'}\ran_g dv_g\right|&\leq& C \sqrt{\log^+(\lambda_{1i'})}, \label{eq:est_mix_w_vi}\\
			\left|\int_{\Sigma}\lan \nabla w_i,\nabla w_{i'}\ran_g dv_g \right| &\leq&\sqrt{\log^+(\lambda_{1i})\log^+(\lambda_{1i'})}. \label{eq:int_nabla_w_i} 
		\end{eqnarray}
		To obtain~\eqref{eq:est_mix_w_vi}, we  first calculate the average of $v_i$ over $\Sigma$. 
		We just consider the case $\lambda_{2i}\neq 0$ here; otherwise $v_i\equiv0.$ 		
        By direct calculation and Cauchy-Schwarz inequality, we have 
		\begin{equation}\label{eq:est_v_i}
			\begin{split}
				&\int_{\Sigma} |v_i+4\log^+(\lambda_{2i})|^2 dv_g 
				= \int_{\Sigma} \left(  \log\Big( \sum_{j=1}^{K_i}t^i_j  \frac{\max\{1, \lambda_{2i}^4\} }{(1+ \lambda_{2i}^2d_g^{\,2}(\cdot,x_j^i))^2}\Big)\right)^2dv_g \leq C,\\
                &\left|\int_\Sigma v_i+4\log^+(\lambda_{2i})dv_g\right| \le C
			\end{split}
		\end{equation}
		where $C>0$ is a constant.
        Since $\int_\Sigma (v_i-\overline{v}_i)^2 dv_g=\min_{a\in\R} \int_\Sigma (v_i-a)^2 dv_g$, the above estimates imply that 
        \begin{equation}
			\label{eq:bar_vi}\overline{v}_i=  -4\log^+(\lambda_{2i})+\mathcal{O}(1),
		\end{equation}
		uniformly for $\lambda_{2i}\ge 0$. 
        
		By straightforward calculation,  in view of $\|\psi_j\|^2 =\lambda_j \|\psi_j\|_{L^2(\Sigma)}^2= 1$ and $\sum_{j=1}^{I} X^2_{ij}=1$ for $i=1,2$,  it holds that 
		\begin{eqnarray*}
			&&\int_{\Sigma}\lan \nabla v_i, \nabla w_{i'}\ran_g  dv_g =\sqrt{\log^+(\lambda_{1i'})} \sum_{j=1}^{I} X_{i'j} \int_{\Sigma} \lan \nabla v_i, \nabla \psi_j\ran_g dv_g \\
			&=& \sqrt{\log^+(\lambda_{1i'})} \sum_{j=1}^{I} X_{i'j}\lambda_j \int_{\Sigma} (v_i-\overline{v}_i)\psi_j dv_g\\
			&\leq & C \sqrt{\log^+(\lambda_{1i'})} \sum_{j=1}^{I} X_{i'j}\lambda_j
			\| v_i-\overline{v}_i\|_{L^2(\Sigma)}\|\psi_j\|_{L^2(\Sigma)}\\
			&\leq&C \sqrt{\log^+(\lambda_{1i'})}\| v_i-\overline{v}_i\|_{L^2(\Sigma)}\stackrel{\eqref{eq:est_v_i}}{\leq} C\sqrt{\log^+(\lambda_{1i'})},
		\end{eqnarray*}
		for some constant $C>0.$

        In view of $\|\psi_j\|=1$ and $\sum_{j=1}^I X_{ij}^2=1$ for $i=1,2$, by Cauchy-Schwarz inequality, we derive that 
		\begin{equation*}
			\begin{split}
				|\lan w_i, w_{i'}\ran| 
				&\leq  \sqrt{\log^+(\lambda_{1i})\log^+(\lambda_{1i'})}\Big( \sum_{j=1}^{I} X_{ij}^2 \int_{\Sigma}\left|\nabla \psi_j\right|_g^2 dv_g\Big)^{\frac 12 } \Big( \sum_{j=1}^{I} X_{i'j}^2 \int_{\Sigma}\left|\nabla \psi_j\right|_g^2 dv_g\Big)^{\frac 12 }\\
				&\leq\sqrt{\log^+(\lambda_{1i})\log^+(\lambda_{1i'})}.
			\end{split}
		\end{equation*}
		
		Applying the estimates~\eqref{eq: est_mix_v12}-\eqref{eq:int_nabla_w_i}, 
		we obtain
		\begin{eqnarray}\label{eq:int_Q}
			&&\int_{\Sigma} Q(\varphi_1,\varphi_2) dv_g \leq \sum_{i=1}^2 \sum_{j=1}^{K_i} 4\pi  (1+o(1))\log^+(\lambda_{2i}) \\
			&&+\frac 1 4 \int_{\Sigma}\left(\sum_{i=1}^2|\nabla w_i|_g^2-\lan\nabla w_1,\nabla w_2\ran_g\right)dv_g+C\sum_{i=1}^2\sqrt{\log^+(\lambda_{1i})},\nonumber
		\end{eqnarray}
		where $\varphi=(\varphi_1,\varphi_2)$ given by \eqref{def:zeta}.
       By the Poincar\'e inequality and \eqref{eq:int_nabla_w_i}, we have $\|w_i\|_{L^2(\Sigma)} \le C \sqrt{\log^+(\lambda_{1i})}$.
       Hence, by the definition of $Q_0$ and \eqref{eq:est_v_i}, we obtain
		\begin{eqnarray}\label{eq:int_beta_varphi}
			\beta \int_{\Sigma} Q_0(\varphi_1,\varphi_2) dv_g 
			=   \frac{\beta}{4}\int_{\Sigma}\left (\sum_{i=1}^2|w_i|^2-w_1w_2\right) dv_g + \mathcal{O}\left(\sum_{i=1}^2\sqrt{\log^+(\lambda_{1i})}+1\right).
		\end{eqnarray}
        Then, by the definition of $w_i$, we have
        \begin{equation}\begin{aligned}\label{eq:int_quadratic}
			&\int_{\Sigma}\Big(\sum_{i=1}^2(|\nabla w_i|_g^2+\beta |w_i|^2)-\lan\nabla w_1,\nabla w_2\ran_g -\beta w_1w_2\Big) dv_g \\
            =& \sum_{j=1}^I \left( 1+\frac{\beta}{\lambda_j}\right) \Big(\sum_{i=1}^2 \log^+(\lambda_{1i})X_{ij}^2 -\sqrt{\log^+(\lambda_{11})\log^+(\lambda_{12})}X_{1j}X_{2j}\Big).
		\end{aligned}\end{equation}
        Combining \eqref{eq:int_Q}, \eqref{eq:int_beta_varphi}, \eqref{eq:int_quadratic}, and using  the inequality $a^2+b^2-ab \geq \frac  1 2(a^2+b^2)$, we obtain 
		\begin{equation}\begin{aligned}\label{eq:Drich_0}
			&\int_{\Sigma} Q(\varphi_1,\varphi_2) + \beta Q_0(\varphi_1,\varphi_2) dv_g\\
			&\leq   \sum_{i=1}^2 \sum_{j=1}^{K_i} 4\pi  (1+o(1))\log^+(\lambda_{2i})
            -\frac18 \sum_{j=1}^I \left(\frac{-\beta}{\lambda_j}-1\right) \sum_{i=1}^2 X_{ij}^2\log^+(\lambda_{1i}) \\
			&+C\sum_{i=1}^2\sqrt{\log^+(\lambda_{1i})}+C.
		\end{aligned}\end{equation}

		It remains to analyze the exponential term. We observe that
		$\|w_i\|_{L^{\infty}(\Sigma)}\leq C\|w_i\|_{L^2(\Sigma)}\leq C\sqrt{\log^+(\lambda_{1i})},$
		as $w_i$ lies in a finite-dimensional subspace spanned by $\{ \psi_j  \}_{j=1}^I$. 
		Moreover,
		$-\overline{v}_i= 4\log^+(\lambda_{2i}) +\mathcal{O}(1)$. 
By standard computations (\cite{battaglia_jevnikar_ruiz2015}), for all $\lambda_{2i} \ge 0$,
        \[
            \log \int_{\Sigma} V_i e^{v_i} dv_g \ge -2\log^+(\lambda_{2i})-C.
        \]
		and it follows that 
		\begin{equation}\begin{aligned}\label{eq:exponential}
			\sum_{i=1}^2\rho_i\log \int_{\Sigma} V_i e^{\varphi_i}&\geq  -\sum_{i=1}^2 \rho_i\bar{v}_i-2\sum_{i=1}^2\rho_i\log^+(\lambda_{2i})-C\sum_{i=1}^2\sqrt{ \log^+(\lambda_{1i})} -C\\
			&\geq  2\sum_{i=1}^2\rho_i\log^+(\lambda_{2i})
			-C\sum_{i=1}^2\sqrt{ \log^+(\lambda_{1i})}-C,
		\end{aligned}\end{equation}
		for some constant $C>0.$

        Combining~\eqref{eq:Drich_0} and \eqref{eq:exponential}, we deduce that 
        \begin{equation}\label{eq:energy0}
            \begin{split}
				J_{\beta,\rho}(\varphi_1,\varphi_2)&\leq \sum_{i=1}^2\left(4\pi K_i-2 \rho_i+o(1)\right)\log^+(\lambda_{2i}) -\frac18 \sum_{j=1}^I \left(\frac{-\beta}{\lambda_j}-1\right) \sum_{i=1}^2 X_{ij}^2\log^+(\lambda_{1i})\\
				&+\mathcal{O}\left( \sum_{i=1}^2\sqrt{\log^+(\lambda_{1i})}+1\right).
			\end{split}
        \end{equation}
        If $I=0$, then $\lambda_{11}=\lambda_{12}=0$ and $\lambda_{21}+\lambda_{22}=\Lambda$.
        Since $4\pi K_i-2 \rho_i<0$, $i=1,2$, the right-hand side of \eqref{eq:energy0} tends to $-\infty$ as $\Lambda \to +\infty$.
        
        If $I>0$, then $-\frac{\beta}{\lambda_j}-1 \ge -\frac{\beta}{\lambda_I}-1>0$ for $j=1,...,I$.
        Since $\sum_{j=1}^I X_{ij}^2=1$, there exists $c>0$ such that
        \[
            J_{\beta,\rho}(\varphi_1,\varphi_2) \le -c\sum_{i=1}^2 \big( \log^+(\lambda_{1i})+\log^+(\lambda_{2i}) \big) + \mathcal{O}(\sqrt{\log\Lambda}).
        \]
        Since $\sum_{i,i'=1}^2 \lambda_{ii'}=\Lambda$, the right-hand side of \eqref{eq:energy0} tends to $-\infty$.
	\end{proof}



We now introduce a two-scale version of the test functions in Lemma \ref{lemma:map}. 
The barycenter components will depend on a parameter $\Lambda$, while the negative eigenspace components will be fixed at a scale $\Lambda_0$.


\begin{proposition}\label{prop:map}
Let $K_1,K_2,I \in \N_0$ with $K_1+K_2+I>0$. Assume that $\rho=(\rho_1,\rho_2) \in \R^2_{>0}\setminus \Gamma$, $\rho_i \in  (2\pi K_i, 2\pi(K_i+1))$, $i=1,2$, and $\beta\in (-\lambda_{I+1},-\lambda_{I})$.
Then, for every $L>0$, there exists $\Lambda_0^*=\Lambda_0^*(L)>0$ such that, for every fixed $\Lambda_0 \ge \Lambda_0^*$ and every $\Lambda \ge \Lambda_0$, 
the mapping $\Phi^\Lambda:(\bga_1)_{\partial,K_1,I}\star (\bga_2)_{\partial, K_2,I}\rightarrow \oH\times\oH$ defined by
\begin{equation*}
\zeta=(X_1, \sigma_1, X_2, \sigma_2; s)\mapsto\left(\begin{array}{c}
	\Phi_1^\Lambda \\
	\Phi_2^\Lambda
\end{array}\right):=\left(\begin{array}{c}
	(v_1^\Lambda-\overline{v_1^\Lambda} + w_1^{\Lambda_0})-\frac 1 2 (v_2^\Lambda-\overline{v_2^\Lambda}+w_2^{\Lambda_0} )\\
	(v_2^\Lambda-\overline{v_2^\Lambda}+w_2^{\Lambda_0})-\frac 1 2 (v_1^\Lambda-\overline{v_1^\Lambda} + w_1^{\Lambda_0})
\end{array}\right).
\end{equation*}
satisfies $\Phi^\Lambda((\bga_1)_{\partial,K_1,I}\star (\bga_2)_{\partial, K_2,I}) \subset J_{\beta,\rho}^{-L}$.
\end{proposition}
\begin{proof}
By Lemma \ref{lemma:map}, for every $L >0$, there exists $\Lambda_0^*(L)$ such that $J_{\beta,\rho}(\varphi_1^{\Lambda'}(\zeta),\varphi_2^{\Lambda'}(\zeta)) \le -L$ for every $\Lambda' \ge \Lambda_0^*(L)$ and every $\zeta' \in (\bga_1)_{\partial,K_1,I}\star (\bga_2)_{\partial, K_2,I}$.

Fix $\Lambda_0 \ge \Lambda_0^*(L)$ and $\Lambda \ge \Lambda_0$.
For $\zeta=(X_1,\sigma_1,X_2,\sigma_2;s)$, set 
\[
\Lambda':=\Lambda_0(s_1+s_3)+\Lambda(s_2+s_4), \quad s':=\frac{1}{\Lambda'}(\Lambda_0 s_1,\Lambda s_2, \Lambda_0 s_3, \Lambda s_4).
\]
Then $s' \in \Delta^3$ and $\Lambda' \ge \Lambda_0 \ge \Lambda_0^*(L)$.
By construction, $\Phi_i^\Lambda(\zeta)=\varphi_i^{\Lambda'}(X_1,\sigma_1,X_2,\sigma_2;s')$, $i=1,2$.
Hence, by Lemma \ref{lemma:map}, $J_{\beta,\rho}(\Phi_1^{\Lambda}(\zeta),\Phi_2^{\Lambda}(\zeta)) \le -L$ uniformly in $\zeta$.  This completes the proof.
\end{proof}

In order to construct a map from the sublevel set $J^{-L}_{\rho,\beta}$ to $(\bga_1)_{\partial, K_1,I}\star (\bga_2)_{\partial, K_2,I}$, we have to introduce an improved Moser-Trudinger inequality.

Firstly, we introduce a local version of Moser-Trudinger inequality, which in the singular setting is also known as Troyanov inequality. Here we use the fact that we have positive weights. 
\begin{lemma}\label{lem:local_MT}
Let $\delta>0$ be sufficiently small. Then: 
\begin{itemize}
\item if  $\Omega_1 \subseteq \Omega_2 \subseteq \Sigma$ is such that $ d_g(\Omega_2, \partial\Sigma)\geq \frac \delta 2 $ and  $d_g\left(\Omega_1, \partial \Omega_2\right) \geq$ $\delta$, then for any $\varepsilon>0$ there exists a constant $C=C(\varepsilon, \delta)$ such that for any $ u=(u_1,u_2)\in\oH\times\oH,$
\begin{equation}
	\label{eq:MT_loc}
	\begin{gathered}
		4 \pi\sum_{i=1}^2 \log \int_{\Omega_1} V_ie^{u_i-\dashint_{\Omega_2} u_i} d v_g \leq 
		\int_{\Omega_2} Q\left(u_1, u_2\right) d v_g+\varepsilon \int_{\Sigma}\sum_{i=1}^2\left|\nabla u_i\right|_g^2 d v_g+C;
	\end{gathered}
\end{equation}
\item  if $\Omega_1\subset \Omega_2 \subset \Sigma$ with $ \Omega_1 \cap \partial\Sigma\neq \emptyset$ and $d_g( \Omega_1\setminus\partial\Sigma, \partial\Omega_2\setminus \partial\Sigma)\geq \delta$, then
for any $\varepsilon>0$ there exists a constant $C=C(\varepsilon, \delta)$ such that for any $ u=(u_1,u_2)\in\oH\times\oH,$
\begin{equation*}
	\begin{split}
		2 \pi\sum_{i=1}^2 \log \int_{\Omega_1} V_ie^{u_i-\dashint_{\Omega_2} u_i} d v_g \leq 
		\int_{\Omega_2} Q\left(u_1, u_2\right) d v_g+\varepsilon \int_{\Sigma}\sum_{i=1}^2\left|\nabla u_i\right|_g^2 d v_g+C .
	\end{split}
\end{equation*}
\end{itemize}
\end{lemma}
The interior case of Lemma  \ref{lem:local_MT} is given by  \cite[Proposition 2.3]{MalchiodiRuiz2013}, while the boundary counterpart can be derived analogously using the Moser-Trudinger inequality \eqref{eq:MT-b}  with suitable cut-off functions. The details are omitted here for brevity.

The following lemma is the so-called  improved Moser-Trudinger inequality in our setting. The case when $I=0$ was proved in \cite{battaglia_jevnikar_ruiz2015}. 
Here, we have to deal with the negative projection $\Pi_I$ and boundary case. 
\begin{lemma}\label{lem:almost_MT}
Let $\delta>0, \theta>0, K_1, K_2 \in \N_0$ and $k_1, k_1^\prime, k_2, k_2^\prime\in \N_0$ with $2k_i+k'_i=K_i$ for $i=1,2.$ Suppose that  $\Omega_{i, j}, \Omega^\prime_{i, j'}
\subset \Sigma$ for 
$i=1,2,j=1, \cdots, k_i, j'=1, \cdots, k_i^\prime$ satisfy 
\begin{equation*}
\left\{\begin{aligned}
	d_g\left(\Omega_{i, j}, \Omega_{i, j^{\prime}}\right) \geq \delta &\quad \forall j, j^{\prime} =1, \cdots, k_i \text { with } j\neq j^{\prime}, \\
	d_g\left(\Omega^{\prime}_{i, j}, \Omega^{\prime}_{i, j^{\prime}}\right) \geq \delta &\quad \forall j, j^{\prime} =1, \cdots, k_i^\prime \text { with } j\neq j^{\prime}, \\
	d_g\left(\Omega_{i, j}, \Omega^\prime_{i, j^{\prime}}\right) \geq \delta &\quad \forall j=1, \cdots, k_i, j^{\prime} =1,\cdots, k_i^\prime, \\
	d_g(\Omega_{i,j}, \partial\Sigma)\geq \frac \delta 2 &\quad \forall i=1,2, j=1,\cdots, k_i,\\
	\Omega^\prime_{i,j^\prime}\cap \partial\Sigma \neq \emptyset & \quad \forall i=1,2, j^\prime= 1, \cdots,k^\prime_i,
\end{aligned}\right.
\end{equation*}
Then, for any $\epsilon>0$, there exists $C$ such that for any $u=\left(u_1, u_2\right) \in\oH \times \oH$ satisfying
$$
\begin{array}{ll}
\int_{\Omega_{i, j}} V_i e^{u_i} d v_g \geq \theta \int_{\Sigma}V_i e^{u_i} d v_g & \forall i\in \{1,2\}, j \in\{1, \cdots, k_i\}, \\
\int_{\Omega^\prime_{i, j}} V_i e^{u_i} d v_g \geq \theta \int_{\Sigma} V_i e^{u_i} d v_g & \forall i\in \{1,2\},  j \in\{1, \cdots,k_i^\prime\}, 	\\
\text{ and }\quad  	\|\Pi_I u_i\|\leq 2 &\forall i\in\{1,2\},
\end{array}
$$
we have 
\begin{equation}\label{eq:inptoved_MT_proj}	\sum_{i=1}^2 2\pi  K_i \log \int_{\Sigma} V_ie^{u_i} dv_g
\leq (1+\epsilon)  \int_{\Sigma}(Q(u_1,u_2)+\beta Q_0(u_1,u_2) )dv_g +C.
\end{equation}
\end{lemma}
\begin{proof}
We only need to consider the case $I>0$. Assume first that $K_1+K_2>0$.
We will consider the estimate by dividing $u_i=u^1_i+u^2_i+u^3_i$ into three parts by truncation in Fourier models:
\begin{eqnarray*}
u^1_i:=\sum_{j=1}^I \lan u_i, \psi_j\ran_{\oH}\psi_j,\quad
u^2_i:=\sum_{j=I+1}^{N_{\varepsilon}-1} \lan u_i, \psi_j\ran_{\oH}\psi_j, \quad  
u^3_i: =\sum_{j=N_{\varepsilon}}^{+\infty} \lan u_i, \psi_j\ran_{\oH}\psi_j,
\end{eqnarray*}
where $N_{\varepsilon}$ sufficiently large will be chosen later. 
By a covering lemma (see \cite[Lemma 3.2]{Malchiodi2007SomeER}, \cite[Lemma 2.4]{MalchiodiRuiz2013} and \cite[Lemma~4.1]{battaglia_jevnikar_ruiz2015})  there exist constants  $\bar{\delta}>0$ and $ \bar{\theta}>0$, independent of $u_i$, and families $\{\Omega_n\}_{n=1}^k$, $\{\Omega^\prime_{n'}\}_{n'=1}^{k'}$ with $k=\max\{ k_1, k_2\}$ and $k^\prime=\max\{k_1^\prime, k_2^\prime\}$ such that
\begin{equation*}
\left\{\begin{aligned}
	d_g\left(\Omega_n, \Omega_{n^{\prime}}\right) \geq 2 \bar{\delta} &&\quad \forall n, n^{\prime} =1, \cdots, k\text { with } n \neq n^{\prime},\\
	d_g\left(\Omega^\prime_n, \Omega^\prime_{n^{\prime}}\right) \geq 2 \bar{\delta} &&\quad \forall n, n^{\prime} =1, \cdots, k^\prime\text { with } n \neq n^{\prime},\\
	d_g\left(\Omega_n, \Omega^\prime_{n^{\prime}}\right) \geq 2 \bar{\delta} &&\quad \forall n =1, \cdots, k, n'=
	1, \cdots, k'	\\
	d_g(\Omega_n, \partial\Sigma)\geq \bar{\delta}&&\quad \forall n=1, \cdots, k, \\
	\Omega^\prime_n \cap  \partial\Sigma \neq\emptyset && \quad \forall n=1, \cdots, k^\prime, 
\end{aligned}\right.
\end{equation*}
and \begin{equation*}
\begin{split}
	\left|\Omega_j\right|	\geq \bar{\theta} \quad \forall j =1, \cdots, k; & \qquad  |\Omega^\prime_{j'}| \geq \bar{\theta} \quad \forall  j'=1, \cdots, k^\prime; \\
	\frac{\int_{\Omega_j} V_i e^{u_i} d v_g}{\int_{\Sigma}V_i e^{u_i} d v_g}  \geq \bar{\theta} \quad \forall  i=1,2, j=1, \cdots, k; & \qquad \frac{\int_{\Omega^{\prime}_{j'}} V_i e^{u_i} d v_g}{\int_{\Sigma} V_i e^{u_i} d v_g} \quad \forall i=1,2, j'=1,\cdots, k^\prime.
\end{split}
\end{equation*}

Assume that $k_2\geq k_1.$
For $j=1,\cdots,k_1$, by applying Lemma \ref{lem:local_MT} for $u^3_i$ with $\Omega_1=\Omega_{j}$ and $\Omega_2=\tilde{\Omega}_{j}:=U_{\frac 1 2 {\bar{\delta}}}(\Omega_{j}):= \{x\in \Sigma: d_g(\{x\}, \Omega_{j})<\frac{\bar{\delta}}{2}  \}$, we have that, for $i=1,2$,
\begin{eqnarray}
&&\sum_{i=1}^2\log \int_{\Sigma} V_ie^{u_i} dv_g  \leq	\sum_{i=1}^2\log \int_{\Omega_{j}} V_ie^{u_i} dv_g+ 2 \log \Big(\frac 1 {\bar{\theta}}\Big)\label{eq:int2vew}\\
&\leq& 
\sum_{i=1}^2 \left( \log \|V_i\|_{L^{\infty}(\Sigma)}+ \|u^1_i\|_{L^{\infty}(\Sigma)}+\|u^2_i\|_{L^{\infty}(\Sigma)} + \Big|\frac 1 {|\tilde{\Omega}_{j}|_g} \int_{\tilde{\Omega}_{j}} u^3_i dv_g \Big| \right) \nonumber\\
&&+ \sum_{i=1}^2  \log \int_{\Omega_{j}} e^{u^3_i-\dashint_{\tilde{\Omega}_{j}} u^3_i} dv_g+2\log\Big(\frac 1 {\bar{\theta}}\Big)\nonumber\\
&\leq&  \sum_{i=1}^2 \left( \|u^1_i\|_{L^{\infty}(\Sigma)}+\|u^2_i\|_{L^{\infty}(\Sigma)} + \Big|\frac 1 {|\tilde{\Omega}_{j}|_g} \int_{\tilde{\Omega}_{j}} u^3_i dv_g \Big| \right) \nonumber\\
&&+\frac 1 {4\pi }\int_{\tilde{\Omega}_j}Q(u_1^3,u_2^3) dv_g +\frac{\varepsilon}{4\pi} \sum_{i=1}^2 \int_{\Sigma}|\nabla u^3_i|^2_g dv_g  +2\log\left(\frac 1 {\bar{\theta}}\right)+C.\nonumber
\end{eqnarray}
Since $u^1_i$ and $u^2_i$ are in finite-dimensional spaces, all norms in finite-dimensional space are equivalent. We can calculate the term involved with $u^1_i$ or $u^2_i$ easily as follows: 
\begin{eqnarray*}
\|u^1_i\|_{L^{\infty}(\Sigma)}\leq C \|u^1_i\|_{L^2(\Sigma)}\leq C \|\Pi_I u_i\|\leq C. 
\end{eqnarray*}
Since $u^2_i$ has no low Fourier coefficients term (more precisely, $\lambda_j+\beta>0$ for $j>I$), the norm $\sqrt{ \int_{\Sigma}|\nabla v|_g^2+\beta |v|^2 dv_g }$ can be applied. 
Hence, by the orthogonal property of eigenfunctions and $\|\Pi_I u_i\|\leq 2$, $i=1,2$, it follows that  
\begin{eqnarray}
&&\sum_{i=1}^2 \|u^2_i\|_{L^{\infty}(\Sigma)}\leq C\sum_{i=1}^2 \sqrt{ \int_{\Sigma}\left(|\nabla u^2_i|_g^2+\beta |u^2_i|^2\right) dv_g}\label{eq:u_12_inf} \\
&\leq& 
\varepsilon\sum_{i=1}^2\int_{\Sigma}\left(|\nabla u^2_i|_g^2+\beta |u^2_i|^2\right) dv_g+\frac C {\varepsilon}\nonumber\\
&\leq& 6 \varepsilon \int_{\Sigma}Q(u_1^2, u_2^2)+ \beta Q_0(u_1^2,u_2^2)  dv_g +\frac{C}{\varepsilon}\nonumber\\
&\leq&6 \varepsilon \int_{\Sigma} Q(u_1, u_2)+\beta Q_0(u_1,u_2) dv_g+\frac{C}{\varepsilon},\nonumber 
\end{eqnarray}
The H\"{o}lder inequality and Poincar\'{e} inequality yield that
\begin{equation}
\label{eq:av_u3} \left|\int_{\tilde{\Omega}_j} u^3_i dv_g\right|\leq \|u^3_i\|_{L^1(\Sigma)}\leq C \int_{\Sigma}|\nabla u^3_i|_g dv_g \leq \varepsilon \int_{\Sigma} |\nabla u^3_i|_g^2 dv_g+ \frac C{\varepsilon}. 
\end{equation}
Next, we consider the case $j=k_1+1,\cdots, k_2$. 
Applying~\eqref{eq:MT_loc},~\eqref{eq:u_12_inf} and \eqref{eq:av_u3}, we deduce that 
\begin{eqnarray}
&&	\log \int_{\Sigma} V_2e^{u_2} dv_g \leq \log \int_{\Omega_j} V_2e^{u_2} dv_g+ \log \left(\frac 1 {\bar{\theta}}\right) \label{eq:int1vew}\\
&\leq& \log \|V_2\|_{L^{\infty}(\Sigma)}+ \|u^1_2\|_{L^{\infty}(\Sigma)}+\|u^2_2\|_{L^{\infty}(\Sigma)} + \left|\frac 1 {|\tilde{\Omega}_j|_g} \int_{\tilde{\Omega}_j} u^3_2 dv_g\right|\nonumber\\
&&
+   \log \int_{\Omega_j} e^{u^3_2-\dashint_{\tilde{\Omega}_j} u^3_2} dv_g+\log\left(\frac 1 {\bar{\theta}}\right)\nonumber\\
&\leq& \frac 1 {4\pi}\int_{\tilde{\Omega}_j}Q(u_1^3,u_2^3) dv_g +\frac{\varepsilon(1+4\pi)}{4\pi}\sum_{i=1}^2 \int_{\Sigma}|\nabla u^3_i|^2_g dv_g\nonumber\\
&&-\log\int_{\Omega_j}  e^{u^3_1-\dashint_{\tilde{\Omega}_j} u^3_1} dv_g+6 \varepsilon \int_{\Sigma}\left(Q\left(u_1, u_2\right)+\beta Q_0(u_1,u_2)\right)dv_g+C.\nonumber
\end{eqnarray}
Concerning that $|\Omega_j|_g\geq \bar{\theta}$ and $ \int_{\Omega_j} ( u_1^3 -\dashint_{\Omega_j} u_1^3) dv_g=0$, we can obtain that 
\[ \log\int_{\Omega_j}  e^{u^3_1-\dashint_{\tilde{\Omega}_j} u^3_1} dv_g\geq \log|\Omega_j| +\dashint_{\Omega_j} u^3_1 - \dashint_{\tilde{\Omega}_j} u^3_1\geq \log \bar{\theta}- \varepsilon \int_{\tilde{\Omega}_j} |\nabla u_1^3|^2 -C_\varepsilon,\]
via the Jensen's inequality for every $\varepsilon>0.$
Since $\tilde{\Omega}_j\cap \tilde{\Omega}_{j'}=\emptyset$ for any $j\neq j',$
we sum the  inequalities~\eqref{eq:int2vew} and \eqref{eq:int1vew} over 
$j =1, \cdots, k_i$ respectively to obtain the following estimate: 
\begin{equation}\begin{aligned}\label{eq:MT-01}
	\quad\sum_{i=1}^2 4\pi k_i \log \int_{\Sigma} V_ie^{u_i} dv_g\leq & \int_{\cup_{j=1}^k \tilde{\Omega}_j} Q(u_1^3,u_2^3) dv_g + C \varepsilon \sum_{i=1}^2 \int_{\Sigma} |\nabla u^3_i|_g^2 dv_g\\
	&+C\varepsilon  \int_{\Sigma}\left(Q\left(u_1, u_2\right)+\beta Q_0 (u_1,u_2)\right) dv_g + C.
	\end{aligned}\end{equation}
	By the same deduction, using the local version Moser-Trudinger inequality of the boundary case from Lemma \ref{lem:local_MT}, we have 
	\begin{equation}\begin{aligned}\label{eq:MT-02}
	\sum_{i=1}^2 2\pi k_i^\prime \log \int_{\Sigma} V_ie^{u_i} dv_g\leq &\int_{\cup_{j=1}^{k'} \tilde{\Omega}^\prime_j} Q(u_1^3,u_2^3) dv_g +  C \varepsilon \sum_{i=1}^2 \int_{\Sigma} |\nabla u^3_i|_g^2 dv_g\\
	&+C\varepsilon  \int_{\Sigma}\left(Q\left(u_1, u_2\right)+\beta Q_0(u_1,u_2)\right) dv_g + C,
	\end{aligned}\end{equation}
	where $\tilde{\Omega}^\prime_j:= U_{\frac  12 {\bar{\delta}}}(\Omega^\prime_j)$ for $j=1, \cdots, k'.$
	
	Combining the estimates \eqref{eq:MT-01} and \eqref{eq:MT-02}, we obtain that 
	\begin{equation*}\begin{aligned}
	\sum_{i=1}^2 2\pi K_i  \log \int_{\Sigma} V_ie^{u_i} dv_g\leq & \int_{\Sigma} Q(u_1^3,u_2^3) dv_g +  C\varepsilon \sum_{i=1}^2 \int_{\Sigma} |\nabla u^3_i|_g^2 dv_g\\
	&+C\varepsilon  \int_{\Sigma}\left(Q\left(u_1, u_2\right)+\beta Q_0(u_1,u_2)\right) dv_g + C,
	\end{aligned}\end{equation*}
	in view of $K_i=2k_i+k_i'$ for $i=1,2.$
	Observe that 
	$\int_{\Sigma}|u^3_i|^2 dv_g \leq \frac 1 {\lambda_{N_{\varepsilon}}}\int_{\Sigma} |\nabla u^3_i|^2_g dv_g$
	and 
	$\sum_{i=1}^2 \int_{\Sigma} |\nabla u^3_i|^2_g dv_g\leq 6\int_{\Sigma} Q(u_1^3,u_2^3) dv_g .$
	Choosing $N_{\varepsilon}$ sufficiently large such that $\frac{4|\beta|}{\lambda_{N_{\varepsilon}}}<\varepsilon$,
	it follows that  
	\begin{eqnarray*}&&
\left| \frac{\beta}{3} \int_{\Sigma}	\Big(\sum_{i=1}^2 |u^3_i|^2+ u^3_1 u^3_2 \Big) dv_g\right|\leq \frac{2|\beta|}{3}\sum_{i=1}^2\int_{\Sigma}|u^3_i|^2 dv_g
\\
&\leq& \frac{2|\beta|}{3\lambda_{N_{\varepsilon}}}\sum_{i=1}^2\int_{\Sigma}|\nabla u^3_i|_g^2 dv_g	\leq  \frac{4|\beta|}{\lambda_{N_{\varepsilon}}}\int_{\Sigma}Q(u_1^3,u_2^3) dv_g \\
&\leq& \varepsilon \int_{\Sigma} Q(u_1^3,u_2^3) dv_g.
\end{eqnarray*}
Using the orthogonal property of the basis $\{\psi_j: j=1,2,\cdots\}$ and  $\|\Pi_Iu_i\|\leq 2 $ for $i=1,2$, 
we have 
\begin{eqnarray*}
&& \int_{\Sigma} Q(u_1^3,u_2^3) dv_g + C\varepsilon \sum_{i=1}^2 \int_{\Sigma} |\nabla u^3_i|_g^2 dv_g\\
&\leq& \frac {1+ 2 C\varepsilon} {1-\varepsilon} \int_{\Sigma} \left(Q(u_1^3,u_2^3)+ \beta Q_0(u^3_1,u^3_2)\right) dv_g\\
&\leq&\frac {1+ 2 C\varepsilon} {1-\varepsilon}\int_{\Sigma}\left(Q\left(u_1, u_2\right)+\beta Q_0(u_1,u_2) \right) dv_g+
C\sum_{i=1}^2\|\Pi_I u_i\|^2\\
&\leq&\frac {1+ 2 C\varepsilon} {1-\varepsilon}\int_{\Sigma}\left(Q\left(u_1, u_2\right)+ \beta Q_0(u_1,u_2) \right) dv_g+C.
\end{eqnarray*}
Hence, 
\begin{eqnarray*}
\quad\sum_{i=1}^2 2\pi K_i \log \int_{\Sigma} V_ie^{u_i} dv_g
\leq \Big(1+\frac{(2C+1)\varepsilon}{1-\varepsilon} \Big) \int_{\Sigma}\left(Q\left(u_1, u_2\right)+\beta Q_0(u_1,u_2)\right) dv_g 
+C.
\end{eqnarray*}
By choosing $\varepsilon>0$ with $\epsilon= \frac{(2C+1)\varepsilon}{1-\varepsilon}$, we conclude the case $K_1+K_2>0$.

Finally, if $K_1=K_2=0$, then the left-hand side of \eqref{eq:inptoved_MT_proj} vanishes. On the other hand, by the orthogonal decomposition with respect to $\{\psi_j\}$, the quadratic form is non-negative for $\psi_j$, $j>I$. 
Hence, by the assumption $\|\Pi_I u_i\| \le 2$, $i=1,2$, we have
\[\int_\Sigma (Q(u_1,u_2)+\beta Q_0(u_1,u_2))dv_g \ge -C,\]
and this completes the proof of Lemma~\ref{lem:almost_MT}.
\end{proof}
The following conclusion then follows by standard arguments (see, e.g., \cite{MalchiodiRuiz2013,Malchiodi2007SomeER,djadli_existence_2008}) and the details are omitted.

\begin{lemma}\label{lem:large_proj_no_concentration_2}
Let $K_1, K_2 \in \N_0$ with $K_1+K_2>0$ and suppose that $2\pi K_i< \rho_i < 2\pi (K_i + 1)$ for $i = 1, 2.$
For any $\varepsilon>0$, there exists $L=L(\varepsilon)>0$ such that for any $u\in J^{-L}_{\beta,\rho}$ at least one of the following statements  holds: 
\begin{eqnarray*}
\dd\left( \frac{ V_i e^{u_i}}{\int_{\Sigma} V_i e^{u_i} dv_g} dv_g , (\Sigma_\partial)_{K_i}\right)\leq \varepsilon, &\text{ for some } i=1,2 \text{ with } K_i>0;\\
\|\Pi_I u_i\|\geq 2, &\text{ for some } i=1,2. 
\end{eqnarray*}
\end{lemma}

When $K_1=K_2=0$, the concentration alternative in Lemma~\ref{lem:large_proj_no_concentration_2} is absent. In this case, since $\rho_i<2\pi$ for $i=1,2$, the Moser--Trudinger inequality, together with the orthogonal decomposition, shows that $J_{\beta,\rho}$ is bounded from below on $\{u:\|\Pi_Iu_i\|<2,\ i=1,2\}$. Hence, for $L>0$ sufficiently large and $u\in J_{\beta,\rho}^{-L}$, $\|\Pi_Iu_i\|\ge 2$ for some $i=1,2$.

\medskip

In the next part, we will construct the map $\Psi$ from the low sublevel set $J^{-L}_{\beta,\rho}$ to $(\bga_1)_{\pa, K_1,I}\star (\bga_2)_{\pa, K_2,I}.$

\begin{proposition}\label{prop:inverse_map}
Let $K_1,K_2, I \in \N_0$ with $K_1+K_2+I>0$.
Assume that $\rho_i \in (2\pi K_i, 2\pi(K_i+1) )$, $i=1,2$, and $-\lambda_{I+1}<\beta<-\lambda_I$  . Then, for $L$ sufficiently large, there exists a continuous map $\Psi: J^{-L}_{\beta,\rho}\rightarrow (\bga_1)_{\pa, K_1,I}\star (\bga_2)_{\pa, K_2,I}$.
\end{proposition}
\begin{proof}
We first consider some special cases. 
If $I=0$, the construction follows by the barycentric projection arguments in \cite{battaglia_jevnikar_ruiz2015} when $K_1,K_2>0$, and by the corresponding one-subcritical argument in \cite{Malchiodi2007SomeER} when exactly one of $K_1,K_2$ is zero. 
The same arguments apply in the present boundary setting using Lemmas~\ref{lem:retraction} and \ref{lem:large_proj_no_concentration_2}.

Hence, we assume $I>0$ below. We first consider the case $K_1, K_2>0$.
Choose $\varepsilon_0>0$ sufficiently small so that, for every $i$ with $K_i>0$, the map 
\begin{equation}\label{eq:retra_phi}
\phi_{K_i}: \{ \sigma\in \cM(\Sigma): \dd(\sigma, (\Sigma_\pa)_{K_i})<\varepsilon_0\}\rightarrow (\Sigma_\pa)_{K_i}
\end{equation} 
given by Lemma~\ref{lem:retraction} is continuous.
Applying  Lemma~\ref{lem:large_proj_no_concentration_2}, we then fix $L:=L(\varepsilon_0/3)$.
For each $i=1,2$, set
\[ d_i(u)= \dd\left( \frac{V_ie^{u_i}}{\int_{\Sigma} V_ie^{u_i} dv_g} dv_g ,(\Sigma_\pa)_{K_i}\right). \]
We introduce the function $\tilde{r}(u)$, defined as $f\left(\frac{d_1(u)}{d_1(u)+d_2(u)}\right),$ where $$f(t)=\max\left\{\min\left\{1, 2t-\frac 1 2\right\}, 0\right\},$$  for $t\in [0,1]$.
This cut-off function $f(t)$ is designed to
nullify the coefficient of $\sigma_i$ when $\sigma_i$ is far away from the barycenter space $(\Si_\pa)_{K_i}$.

For each $i=1,2,$ we define 
\[\tilde{r}_i(u)=  \min\{1,\|\Pi_I (2u_i+u_j)\|\}\in [0,1] \quad \text{where } 1 \le j \neq i \le 2.\]       
Let $$ s(u):=
\left(\begin{array}{c}
s_1(u)\\
s_2(u)\\
s_3(u)\\
s_4(u)
\end{array}\right)=\left(\begin{array}{c}
\frac{\tilde{r}_1(u)}{\tilde{r}_1(u)+\tilde{r}_2(u)}\min\{1,\max\{\tilde{r}_1(u),\tilde{r}_2(u)\}\}\\
(1-\tilde{r}(u))(1-\min\{1,\max\{\tilde{r}_1(u),\tilde{r}_2(u)\}\})\\
\frac{\tilde{r}_2(u)}{\tilde{r}_1(u)+\tilde{r}_2(u)}\min\{1,\max\{\tilde{r}_1(u),\tilde{r}_2(u)\}\}\\
\tilde{r}(u)(1-\min\{1,\max\{\tilde{r}_1(u),\tilde{r}_2(u)\}\})
\end{array}\right).$$

We construct the map from $J^{-L}_{\beta,\rho} $ into $(\bga_1)_{\pa, K_1,I}\star (\bga_2)_{\pa, K_2,I}$ as follows:
\[\Psi: u\mapsto \zeta(u)=( X_1(u),\sigma_1(u),X_2(u),\sigma_2(u); s(u))\]
where for each $i=1,2$ with $K_i>0$,
\[ \sigma_i(u)=(\cP_i)_*\phi_{K_i}\left(\frac{V_ie^{u_i}}{\int_{\Sigma} V_ie^{u_i} dv_g} \right),\]
whenever $d_i(u)<\varepsilon_0$, and
\[ X_i(u)= \frac{\Pi_{I} (2u_i+u_j)}{\|\Pi_{I} (2u_i+u_j)\|}\in \mathbb{S}^{I-1}  \text{ for } 1 \le j \ne i \le 2\]
whenever $\|\Pi_{I} (2u_i+u_j)\|>0$.
Here, $\phi_{K_i}$ is the retraction \eqref{eq:retra_phi} from Lemma~\ref{lem:retraction}.

\medskip

\noindent {\it Claim. $\Psi$ is well-defined and continuous.} \par        
First, we note that if $s_{2i-1}(u)>0$, then $\tilde{r}_i(u)>0$, and hence $X_i(u)$ is well-defined. 
\medskip

\noindent {\it Case I.  $\|\Pi_{I} u_i\|\geq 1$ for either $i=1$ or $i=2$. }\par 
In this case, we have 
\begin{equation}\label{eq:ri_positive}
\max\{\|\Pi_I(2u_1+u_2)\|, \|\Pi_I(2u_2+u_1)\|\} \ge 2\max_{i=1,2} \|\Pi_I u_i\|-\min_{i=1,2} \|\Pi_I u_i\| \ge 1,
\end{equation}
so it holds that $s_1(u)+s_3(u)=1$.
Thus, the coefficients corresponding to $\sigma_1$ and $\sigma_2$ in the expression for $\zeta(u)$ are zero.
Then, $\zeta(u)$ is well-defined. 
\par

\noindent {\it Case II.  $\|\Pi_{I} u_i\|< 1 $ for both $i=1,2$. }\par 
In this case,  Lemma~\ref{lem:large_proj_no_concentration_2} indicates that for some $i=1,2$ with $K_i>0$, we have 
$$ d_i(u) \leq \frac{\varepsilon_0}3. $$
Without loss of generality, we assume $ d_1(u) \leq \frac{\varepsilon_0}3. $ 
If $d_2(u)<\varepsilon_0$,  then both $\sigma_1(u)$ and $\sigma_2(u)$ are well-defined.
Otherwise, by the definition, $d_2(u) \ge \varepsilon_0$, so $\frac{d_1(u)}{d_1(u)+d_2(u)} \leq \frac14$.
Due to the definition of $f$, it follows that $\tilde{r}(u)=0$ and $s_4(u)=0$.
Therefore, the coordinate $\sigma_2(u)$ is irrelevant in the join, and $\zeta(u)$ is well-defined.

The continuity of $\Psi$ follows from the continuity of the above maps and the definition of the join topology, since a coordinate becomes irrelevant whenever its simplex coefficient vanishes.

\medskip

It remains to consider the cases where at least one of $K_1,K_2$ vanishes. Suppose $K_1=0<K_2$. Setting $d_1(u)\equiv\varepsilon_0$ and omitting the first barycenter coordinate, the same construction applies. Indeed, in Case II, Lemma~\ref{lem:large_proj_no_concentration_2} gives $d_2(u)\le\varepsilon_0/3$, and hence $\tilde{r}(u)=1$, so the coefficient corresponding to the absent barycenter component vanishes. The case $K_2=0<K_1$ is analogous.

Suppose next $K_1=K_2=0$. Then $(\bga_1)_{\pa, K_1,I}\star (\bga_2)_{\pa, K_2,I}=S^{I-1}\star S^{I-1}$.
By the observation following Lemma~\ref{lem:large_proj_no_concentration_2} and \eqref{eq:ri_positive}, for $L>0$ sufficiently large, $\tilde{r}_1(u)+\tilde{r}_2(u)>0$ for every $u\in J_{\beta,\rho}^{-L}.$
Hence, using the notation above, we define
\[\Psi(u)= \left( X_1(u),X_2(u); \frac{\widetilde r_1(u)}{\widetilde r_1(u)+\widetilde r_2(u)}, \frac{\widetilde r_2(u)}{\widetilde r_1(u)+\widetilde r_2(u)} \right),\]
which is well-defined and continuous.
\end{proof}

Fix $L>0$ sufficiently large so that $\Psi$ is well-defined on $J_{\beta,\rho}^{-L}$, and let $\Lambda_0^*(L)$ be given by Proposition \ref{prop:map}. Then, for every fixed $\Lambda_0 \ge \Lambda_0^*(L)$ and every $\Lambda \ge \Lambda_0$, it holds that $\Phi^{\Lambda}((\bga_1)_{\pa,K_1,I}\star (\bga_2)_{\pa,K_2,I})\subseteq J^{-L}_{\beta,\rho}$.
Thus, the composition $\Psi\circ\Phi^\Lambda$ is well-defined and continuous.
We next show that, if $\Lambda_0$ is chosen sufficiently large, then $\Psi\circ\Phi^\Lambda$ is homotopic to the identity for all sufficiently large $\Lambda$. This can be carried out since the weight of the points on the boundary is smaller than that of interior points, as explained in the introduction.

\begin{proposition}
Let $K_1,K_2, I \in \N_0$ with $K_1+K_2+I>0$.
Assume that $\rho_i \in (2\pi K_i, 2\pi(K_i+1) )$ for $i=1,2$, and $-\lambda_{I+1}<\beta<-\lambda_I$.
Let $L>0$ be fixed as above. 
Then there exists $\overline{\Lambda}_0=\overline{\Lambda}_0(L) \ge \Lambda_0^*(L)$ such that, for every fixed $\Lambda_0\ge \overline{\Lambda}_0$ and all sufficiently large $\Lambda$, the composition $\Psi\circ\Phi^\Lambda$ is homotopically equivalent to the identity map on $(\bga_1)_{\pa, K_1,I}\star (\bga_2)_{\pa, K_2,I}$, denoted by $\Psi\circ \Phi^{\Lambda} \simeq \id_{(\bga_1)_{\pa, K_1,I}\star (\bga_2)_{\pa, K_2,I}}$.
\end{proposition}
\begin{proof} 
We first consider the case $I=0$. If $K_1,K_2>0$, the conclusion follows by the same barycentric homotopy argument as in \cite{battaglia_jevnikar_ruiz2015}. If exactly one of $K_1,K_2$ is zero, the argument reduces to the corresponding one-subcritical case treated in \cite{Malchiodi2007SomeER}. The same arguments apply in the present boundary setting, with the barycenter spaces and retractions used in Proposition~\ref{prop:inverse_map}. Hence, we assume $I>0$ below.

We now assume $I>0$. We first consider the case $K_1,K_2>0$.

\noindent\textit{Step 1. Limiting maps and reparametrization of simplex coordinate.}

For large $\Lambda$, we write
\[ \Psi\circ\Phi^\Lambda(\zeta)=\bigl(X_1^\Lambda(\zeta),\sigma_1^\Lambda(\zeta),X_2^\Lambda(\zeta),\sigma_2^\Lambda(\zeta);s^\Lambda(\zeta)\bigr)
\]
for $\zeta=\left(X_1, \sigma_1, X_2, \sigma_2;s\right)\in (\bga_1)_{\pa,K_1,I}\star (\bga_2)_{\pa,K_2,I}$ where
\[
s^\Lambda(\zeta)=(s_1^\Lambda(\zeta),\ldots,s_4^\Lambda(\zeta))\in\Delta^3.
\]

Based on the estimates established in the proof of Lemma~\ref{lemma:map}, for each $\zeta$ with $s_2>0$, the corresponding normalized measure $\frac{ V_1 e^{\Phi_1^{\Lambda} (\zeta)}}{ \int_{\Sigma}  V_1 e^{\Phi_1^{\Lambda}}(\zeta) dv_g} dv_g$ approaches the barycenter space $(\Sigma_\partial)_{K_1}$ as $\Lambda\to\infty$, and an analogous statement holds for the case $s_4>0$.
Hence, for all sufficiently large $\Lambda$, the retraction used in the definition of $\Psi$ is applicable, so $\sigma_1(\Phi^\Lambda(\zeta))$ and $\sigma_2(\Phi^\Lambda(\zeta))$ are well-defined when $s_2>0$ and $s_4>0$, respectively. 
The same estimates show that the limits
\[
\sigma_1^\infty(\zeta):=\lim_{\Lambda\to\infty} \sigma_1\bigl(\Phi^\Lambda(\zeta)\bigr), \qquad \sigma_2^\infty(\zeta):=\lim_{\Lambda\to\infty} \sigma_2\bigl(\Phi^\Lambda(\zeta)\bigr)
\]
exist for $s_2>0$ and $s_4>0$, respectively.
Moreover, these convergences are uniform on $\{s_2>\delta\}$ and $\{s_4>\delta\}$ for any fixed $\delta>0$.
It implies that $\sigma_i^\infty(\zeta)$, $i=1,2$ is continuous in $\{s_{2i}>0\}$.
We also have
\begin{equation}\label{eq:limitingsupp}
\text{supp}(\sigma_i^\infty(\zeta)) \subseteq \text{supp}(\sigma_i), \qquad i=1,2.
\end{equation}

We next consider the spherical components $X_i$. 
Denote 
$$A_i^\Lambda(\zeta):=(\langle v_i^\Lambda, \psi_1\rangle_{\oH},\cdots, \langle v_i^\Lambda, \psi_I\rangle_{\oH}) \in \mathbb{R}^I,$$
for $i=1,2$.
By the estimates used in the proof of Lemma~\ref{lemma:map}, there exists $C>0$, independent of $\Lambda$ and $\zeta$, such that $\|A_i^\Lambda(\zeta)\| \le C$.
Moreover, as $\Lambda \to \infty$, 
\[A_1^\Lambda(\zeta) \to A_1^\infty(\zeta) \quad \text{ for } s_2>0, \qquad A_2^\Lambda(\zeta)\to A_2^\infty(\zeta) \quad \text{ for } s_4>0,\]
and the limiting maps are continuous on each region.

Fix $\delta>0$. Since $\|A_i^\infty(\zeta)\|$ is uniformly bounded, there exists $\overline{\Lambda}_0 \ge \Lambda_0^*(L)$ such that, for every fixed $\Lambda_0 \ge \overline{\Lambda}_0$, the limiting maps
\[
X_1^\infty(\zeta):=\frac{A_1^\infty(\zeta)+\sqrt{\log^+(s_1\Lambda_0)}X_1}{|A_1^\infty(\zeta)+\sqrt{\log^+(s_1\Lambda_0)}X_1|}, \quad  \quad 
X_2^\infty(\zeta):=\frac{A_2^\infty(\zeta)+\sqrt{\log^+(s_3\Lambda_0)}X_2}{|A_2^\infty(\zeta)+\sqrt{\log^+(s_3\Lambda_0)}X_2|}\]
are well-defined on $\{s_1 \ge \delta, s_2>0\}$ and $\{s_3 \ge \delta, s_4>0\}$, respectively.
Moreover, $X_1^\Lambda \to X_1^\infty$ and $X_2^\Lambda \to X_2^\infty$ uniformly on $\{s_1 \ge \delta, s_2\ge \delta\}$ and $\{s_3 \ge \delta, s_4 \ge\delta\}$, respectively.
By increasing $\overline{\Lambda}_0$, if necessary, we may further assume that, for every fixed $\Lambda_0 \ge \overline{\Lambda}_0$,
\begin{equation}\label{eq:limitingmap}
\langle X_i^\Lambda(\zeta), X_i \rangle>0, \quad \langle X_i^\infty(\zeta), X_i \rangle>0, \ \  i=1,2
\end{equation}
whenever $s_{2i-1} \ge \delta, s_{2i}>0$.

\medskip

We now reparametrize the simplex coordinates. Fix $\delta=1/8$, and define
\[
q(t):=\left(t-\frac18\right)_+, \qquad \tau_j(s):=\frac{q(s_j)}{\sum_{\ell=1}^4q(s_\ell)}, \quad j=1,\ldots,4.
\]
Since $\max_j s_j\ge1/4$, the denominator is positive and $\tau(s):=(\tau_1(s),...,\tau_4(s))\in \Delta^3$ is well-defined. 
We also note that $\tau_j(s)>0$ implies $s_j>1/8$.
Moreover, $s_j=0$ implies $\tau_j(s)=0$, so $\zeta \mapsto (X_1, \sigma_1, X_2, \sigma_2;\tau(s))$ defines a continuous map from $(\bga_1)_{\pa,K_1,I}\star (\bga_2)_{\pa,K_2,I}$ to itself.

\medskip

Since the limiting maps $X_i^\infty$  are defined only for $s_2>0$ and $s_4>0$, respectively, and need not extend continuously to $s_2=0$ or $s_4=0$, we introduce one more cutoff function.
Let 
\[
\chi(t):=\begin{cases}
0 & 0 \le t \le \tfrac18,\\
8t-1 & \tfrac18<t<\tfrac14,\\
1 & \tfrac14 \le t\le 1,
\end{cases}
\]
and we define the modified spherical maps
\[
\hat{X}_1(\zeta):=\frac{(1-\chi(s_2))X_1^\Lambda(\zeta)+\chi(s_2)X_1^\infty(\zeta)}{|(1-\chi(s_2))X_1^\Lambda(\zeta)+\chi(s_2)X_1^\infty(\zeta)|} \quad \text{for } s_2>0; \quad \hat{X}_1(\zeta):=X_1^\Lambda(\zeta) \quad \text{for } s_2=0
\]
and $\hat{X}_2(\zeta)$ analogously. 
Indeed, by \eqref{eq:limitingmap}, the denominators above do not vanish. 
Moreover, since $\chi(t)=0$ for $t \le 1/8$, we have $\hat{X}_1=X_1^\Lambda$ for $s_2 \le 1/8$ and it is similar for $\hat{X}_2$. 
Hence, $\hat{X}_1$ and $\hat{X}_2$ are continuous on $\{s_1>0\}$ and $\{s_3>0\}$, respectively.

\medskip

\noindent \textit{Step 2. Construction of the homotopy}

Let us define
\[\begin{split}
F_1(\zeta) &:=(X_1^\Lambda,\sigma_1^\Lambda,X_2^\Lambda,\sigma_2^\Lambda;\tau(s)), \\
F_2(\zeta) &:=(\hat{X}_1,\sigma_1^\infty,\hat{X}_2,\sigma_2^\infty;\tau(s)), \\
F_3(\zeta) &:=(X_1, \sigma_1, X_2, \sigma_2;\tau(s))
\end{split}
\]
and we claim $\Psi \circ \Phi^\Lambda \simeq F_1 \simeq F_2 \simeq F_3 \simeq \id_{(\bga_1)_{\pa, K_1,I}\star (\bga_2)_{\pa, K_2,I}}$.

First, it is easy to check $F_3 \simeq \id_{(\bga_1)_{\pa, K_1,I}\star (\bga_2)_{\pa, K_2,I}}$ by deforming the simplex parameter linearly from $s$ to $\tau(s)$, i.e., by the homotopy
\[(\zeta, t) \mapsto (\zeta, t\tau(s)+(1-t)s).\]
Similarly, we have $\Psi \circ \Phi^\Lambda \simeq F_1$.

Next, we show $F_1 \simeq F_2$.
We define 
\[
H^1(\zeta,t):=(Y_1(\zeta,t),\sigma_1^{\frac{\Lambda}{1-t}}(\zeta),Y_2(\zeta,t),\sigma_2^{\frac{\Lambda}{1-t}}(\zeta);\tau(s))
\]
where 
\[
Y_1(\zeta,t)=\frac{(1-\chi(s_2))X_1^\Lambda(\zeta)+\chi(s_2)X_1^{\frac{\Lambda}{1-t}}(\zeta)}{|(1-\chi(s_2))X_1^\Lambda(\zeta)+\chi(s_2)X_1^{\frac{\Lambda}{1-t}}(\zeta)|}, \,  Y_2(\zeta,t)=\frac{(1-\chi(s_4))X_2^\Lambda(\zeta)+\chi(s_4)X_2^{\frac{\Lambda}{1-t}}(\zeta)}{|(1-\chi(s_4))X_2^\Lambda(\zeta)+\chi(s_4)X_2^{\frac{\Lambda}{1-t}}(\zeta)|}.
\]
for $\zeta \in (\bga_1)_{\pa, K_1,I}\star (\bga_2)_{\pa, K_2,I}$, $t\in [0,1)$.
Then it is trivial that $H^1(\zeta,0)=F_1$.
We also note that, for $i=1,2$, 
\[\tau_{2i}(s)>0 \Rightarrow s_{2i}>\tfrac18; \quad \tau_{2i-1}(s)>0, \chi(s_{2i})>0 \Rightarrow s_{2i-1}>\tfrac18, s_{2i}>\tfrac18.\]
Thus, the uniform convergence properties of the limiting maps imply that $H^1$ extends continuously to $t=1$, and we obtain $F_1 \simeq F_2$.

Finally, we show $F_2 \simeq F_3$.
We define
\[
H^2(\zeta,t):=\bigl(\frac{(1-t)\hat{X}_1+tX_1}{|(1-t)\hat{X}_1+tX_1|},(1-t)\sigma_1^\infty+t\sigma_1, \frac{(1-t)\hat{X}_2+tX_2}{|(1-t)\hat{X}_2+tX_2|},(1-t)\sigma_2^\infty+t\sigma_2;\tau(s)\bigr).
\]
By \eqref{eq:limitingsupp}, the measure in the above homotopy belongs to the same barycenter space. Moreover, by \eqref{eq:limitingmap} and the definition of $\hat{X}_i$, it holds that $\langle \hat{X}_i(\zeta), X_i \rangle>0$.
Hence, the denominators above do not vanish.
Thus, $H^2$ is a well-defined homotopy from $F_2$ to $F_3$, and this completes the proof.

\medskip

It remains to consider the cases where at least one of $K_1,K_2$ vanishes. Suppose first that $K_1=0<K_2$. Then $s_2=0$ and $v_1^\Lambda\equiv0$, so the first barycenter component is absent. Moreover, $X_1^\Lambda=X_1$ whenever the corresponding spherical component is relevant. 
Hence the preceding argument applies with all terms involving $\sigma_1$ omitted. The case $K_2=0<K_1$ is analogous.

Finally, if $K_1=K_2=0$, then $v_1^\Lambda=v_2^\Lambda=0$, so $X_i(\Phi^\Lambda(\zeta))=X_i$ whenever the corresponding simplex coefficient is positive. Moreover, for $\Lambda_0$ sufficiently large, the simplex coefficients are well-defined and $s_{2i-1}=0$ implies $\tilde{r}_i(\Phi^\Lambda(\zeta))=0$. Hence, a linear deformation of the simplex coordinates yields the homotopy $\Psi\circ \Phi^{\Lambda} \simeq \id$.
\end{proof}

\begin{remark}
\label{prop:algebra}
Since $\bga_i$ is homeomorphic to $\mathbb{S}$, 
we have $(\bga_i)_{\pa,K_i}$ is homeomorphic to $\mathbb{S}^{2K_i-1}$ if $K_i>0$. 
We also note that $\mathbb{S}^m \star \mathbb{S}^n \cong \mathbb{S}^{m+n+1}$; see, for instance, \cite[Chapter 0]{Hatcher2002}.
With our convention omitting empty factors in the join, we have $(\bga_1)_{\partial,K_1,I}\star (\bga_2)_{\partial,K_2,I} \cong \mathbb{S}^q$
\end{remark}

\section{Existence of solutions}\label{sec:4}
First, when $K_1=K_2=I=0$, there exists a minimizer of $J_{\beta,\rho}$. Next, we will use the min-max scheme to establish the existence of solutions when $K_1+K_2+I>0$. The proof mainly follows~\cite{battaglia_jevnikar_ruiz2015} (also see~\cite{ding1999existence,malchiodi2008morse}, for instance).

Let $K_1, K_2, I \in \mathbb{N}_0$ be determined by $2\pi K_i < \rho_i < 2\pi (K_i+1)$, $i=1,2$ and $-\lambda_{I+1} < \beta <-\lambda_I$.
We assume $K_1+K_2+I>0$, and set 
\[\mathcal{X}:=(\gamma_1)_{\partial,K_1,I}\star (\gamma_2)_{\partial,K_2,I}.\]

Proposition~\ref{prop:inverse_map} shows that the space $\mathcal{X}$ is dominated by the sublevel set $J^{-L}_{\beta,\rho}$. Since  $\mathcal{X}$ is non-contractible from Remark \ref{prop:algebra}, it follows that
$ \Phi^{\Lambda}(\mathcal{X})$ is non-contractible in $ J^{-L}_{\beta,\rho},$ for $\Lambda>0$  sufficiently large. 
We choose $\Lambda>0$ sufficiently large so that $\Phi^{\Lambda}(\mathcal{X})$ is contained in $J^{-2L}_{\beta,\rho}$. 
Subsequently, we define the topological cone 
\[ \cC:= \mathcal{X} \times[0,1]/\sim, \]
where $(x,0)\sim (x', 0)$ for any $x,x'\in \mathcal{X}.$
Since the join $\mathcal{X}=(\gamma_1)_{\pa, K_1,I} \star (\gamma_2)_{\pa, K_2,I}$ is homeomorphic to the sphere $\mathbb{S}^{2(K_1+K_2+I)-1}$, the cone $\mathcal{C}$ is homeomorphic to a ball of dimension $2(K_1+K_2+I)$.
The min-max value is defined as
$$
\bm{c}:= \inf_{\xi \in \cB} \max_{u \in \mathcal{C}} J_{\beta,\rho}(\xi(u)),
$$
where the admissible class $\cB$ is given by
$$
\cB := \left\{ \xi \in  C(\mathcal{C}, \cH) \mid \xi(\zeta,1) = \Phi^{\Lambda}(\zeta) \text{ for } \zeta \in \mathcal{X} \right\}.
$$
Since the map $(\zeta,t) \mapsto t\Phi^{\Lambda}(\zeta)$ belongs to $\cB$, the class $\cB$ is non-empty. Furthermore, we have
$$
\max_{(\zeta,1) \in \partial\mathcal{C}} J_{\beta,\rho}(\xi(\zeta,1))= \max_{ \zeta\in  \mathcal{X}  }J_{\beta,\rho}(\Phi^{\Lambda}(\zeta))\leq -2L
$$
for any $\xi$ in $\cB$.
Since $\partial\mathcal{C}$ is contractible in $\mathcal{C}$,  the boundary map $\xi|_{\partial\mathcal{C}}=\Phi^\Lambda$ is null-homotopic in $\xi(\mathcal{C})$ for any mapping $\xi$ in $\cB$. However, since $\Phi^{\Lambda}$ is not null-homotopic in the sublevel set $J^{-L}_{\beta,\rho}$, it follows that $\xi(\cC) \nsubseteq J^{-L}_{\beta,\rho}.$ Therefore, we deduce that the min-max value $\bm{c}$ satisfies $\bm{c}\geq -L$.

While a standard variational argument (see \cite{struwe2000variational}, for instance) yields that there exists a Palais-Smale sequence at the level $\bm{c}$ for the functional $J_{\beta,\rho}$,  we cannot ensure the existence of a critical point since the Palais-Smale condition's validity remains unknown. We therefore use the {\it monotonicity trick} from Struwe in \cite{struwe1985evolution} to get a sequence of solutions for $\rho_n\to\rho$. The existence of a solution then follows by the compactness result in  Theorem~\ref{thm:cpt_sols}.

\

\noindent {\bf  Acknowledgements.} 
Z.~Hu gratefully acknowledges Professors Thomas Bartsch, Mohameden Ahmedou, and Lei Zhang for their support and helpful discussions. Z.~Hu was partially supported by the National Natural Science Foundation of China and the China Postdoctoral Science Foundation (No.~12601383, \ No.  26Z020703201).
A. Jevnikar is partially supported by INdAM-GNAMPA project \emph{Analisi qualitativa di problemi differenziali non lineari}.
S.-H. Moon was  supported   by the National Research Foundation of Korea(NRF) grant funded by the Korea government(MSIT) (No. RS-2022-NR072398, \ No. RS-2022-NR070702).

\medskip

\noindent\textbf{Competing interests.}
The authors declares no competing interests.

\medskip
\noindent\textbf{Data availability.}
No datasets were generated or analysed during the current study.


\begin{thebibliography}{10}
\bibitem{ABH2026}
Mohameden Ahmedou, Thomas Bartsch, and Zhengni Hu.
\newblock Blow-up solutions for the steady state of the Keller-Segel system on Riemann surfaces.
\newblock {\em Calc. Var. Partial Differ. Equ.}, 65(5):149, 2026.








\bibitem{battaglia_existence_2015}
Luca Battaglia.
\newblock Existence and multiplicity result for the singular {Toda} system.
\newblock {\em J. Math. Anal. Appl.}, 424(1):49--85, 2015.


\bibitem{Battaglia2018}
Luca Battaglia.
\newblock A general existence result for stationary solutions to the
{Keller}-{Segel} system.
\newblock {\em Discrete Contin. Dyn. Syst.}, 39(2):905--926, 2019.


\bibitem{battaglia_jevnikar_ruiz2015}
Luca Battaglia, Aleks Jevnikar, Andrea Malchiodi, and David Ruiz.
\newblock A general existence result for the {T}oda system on compact surfaces.
\newblock {\em Adv. Math.}, 285:937--979, 2015.


\bibitem{battaglia2013moser}
Luca Battaglia and Andrea Malchiodi.
\newblock A {M}oser-{T}rudinger inequality for the singular {T}oda system.
\newblock {\em Bull. Inst. Math. Acad. Sin.}, 9(1):1{\textendash}23, 2014.


\bibitem{Battaglia2015a_note}
Luca Battaglia and Gabriele Mancini.
\newblock A note on compactness properties of the singular {T}oda system.
\newblock {\em Atti Accad. Naz. Lincei Cl. Sci. Fis. Mat. Natur.},
26(Rendiconti Lincei. Mathematics and Applications;vol. 26, Issue
3):299--307, 2015.


\bibitem{bolton1988conformal}
John Bolton, Gary~R. Jensen, Marco Rigoli, and Lyndon~M. Woodward.
\newblock On conformal minimal immersions of {$S^2$} into $
\mathbb{C}\mathbb{P}^n $.
\newblock {\em Math. Ann.}, 279(4):599--620, 1988.


\bibitem{brezis_uniform_1991}
Ha\"{i}m Brezis and Frank Merle.
\newblock Uniform estimates and blow-up behavior for solutions of { $ -{\Delta}
u={V(x)} e^u$} in two dimensions.
\newblock {\em Comm. Partial Differential Equations}, 16(8-9):1223--1253, 1991.


\bibitem{car-mal}
Alessandro Carlotto and Andrea Malchiodi.
\newblock Weighted barycentric sets and singular Liouville equations on compact
surfaces.
\newblock {\em J. Funct. Anal.}, 262(2):409--450, 2012.


\bibitem{chern1955}
Shiing-shen Chern.
\newblock An elementary proof of the existence of isothermal parameters on a
surface.
\newblock {\em Proc. Amer. Math. Soc.}, 6:771--782, 1955.


\bibitem{DAprile_Pistoia_Ruiz2015}
Teresa D'Aprile, Angela Pistoia, and David Ruiz.
\newblock A continuum of solutions for the {S}{U}(3) {T}oda system exhibiting
partial blow-up.
\newblock {\em Proc. Lond. Math. Soc.}, 111(4):797--830, 2015.


\bibitem{D'aprile_asymmetric_2016}
Teresa D'Aprile, Angela Pistoia, and David Ruiz.
\newblock Asymmetric blow-up for the {S}{U}(3) {T}oda system.
\newblock {\em J. Funct. Anal.}, 271(3):495--531, 2016.




\bibitem{ding1999existence}
Weiyue Ding, J{\"u}rgen Jost, Jiayu Li, and Guofang Wang.
\newblock Existence results for mean field equations.
\newblock {\em Ann. Inst. H. Poincar{\'e} Anal. Non Lin{\'e}aire},
16(5):653--666, 1999.










\bibitem{DjadliMalchiodi2008}
Z.~Djadli and A.~Malchiodi.
\newblock Existence of conformal metrics with constant {Q}-curvature.
\newblock {\em Annals of Mathematics (2)}, 168(3):813--858, 2008.


\bibitem{djadli_existence_2008}
Zindine Djadli.
\newblock Existence result for the mean field problem on {R}iemann surfaces of
all genuses.
\newblock {\em Commun. Contemp. Math.}, 10(02):205--220, 2008.


\bibitem{doliwa1997holomorphic}
Adam Doliwa.
\newblock Holomorphic curves and toda systems.
\newblock {\em Lett. Math. Phys.}, 39(1):21--32, 1997.


\bibitem{dunne1995self}
Gerald~V. Dunne.
\newblock {\em Self-dual {C}hern-{S}imons theories}.
\newblock Lecture Notes in Physics. Springer, Berlin, 1995.


\bibitem{dunne1991self}
Gerald~V. Dunne, R.~Jackiw, So-Young Pi, and Carlo~A. Trugenberger.
\newblock Self-dual {Chern}-{Simons} solitons and two-dimensional nonlinear
equations.
\newblock {\em Phys. Rev. D}, 43:1332--1345, 1991.


\bibitem{Esposito2014singular}
Pierpaolo Esposito and Pablo Figueroa.
\newblock Singular mean field equations on compact {R}iemann surfaces.
\newblock {\em Nonlinear Anal.}, 111:33--65, 2014.


\bibitem{GuZhang2020}
Yi~Gu and Lei Zhang.
\newblock Degree counting theorems for singular liouville systems.
\newblock {\em Ann. Sc. Norm. Super. Pisa Cl. Sci. (5)}, 21(Special Issue
1–2):1103--1135, 2020.


\bibitem{Hatcher2002}
Allen Hatcher.
\newblock {\em Algebraic Topology}.
\newblock Cambridge University Press, Cambridge, 2002.




\bibitem{jevnikar_kallel_Malchiodi2015}
Aleks Jevnikar, Sadok Kallel, and Andrea Malchiodi.
\newblock A topological join construction and the {T}oda system on compact
surfaces of arbitrary genus.
\newblock {\em Anal. PDE}, 8(8):1963--2027, 2015.


\bibitem{jost_lin_wang2006}
Jürgen Jost, Changshou Lin, and Guofang Wang.
\newblock Analytic aspects of the {{Toda}} system: {I}{I.} bubbling behavior
and existence of solutions.
\newblock {\em Comm. Pure Appl. Math.}, 58(4):526--558, 2005.


\bibitem{jostandwang2001}
Jürgen Jost and Guofang Wang.
\newblock Analytic aspects of the {{{Toda}}} system: {I.} { A}
{M}oser-{T}rudinger inequality.
\newblock {\em Comm. Pure Appl. Math.}, 54(11):1289--1319, 2001.


\bibitem{lee_degree_2018}
Youngae Lee, Chang-Shou Lin, Juncheng Wei, and Wen Yang.
\newblock Degree counting and {Shadow} system for {Toda} system of rank two:
{One} bubbling.
\newblock {\em J. Differential Equations}, 264(7):4343--4401, 2018.


\bibitem{Lee2020}
Youngae Lee, Chang-Shou Lin, Wen Yang, and Lei Zhang.
\newblock Degree counting for toda system with simple singularity: One point
blow up.
\newblock {\em J. Differential Equations}, 268(5):2163--2209, 2020.


\bibitem{li2005solutions}
Jiayu Li and Yuxiang Li.
\newblock Solutions for {T}oda systems on {R}iemann surfaces.
\newblock {\em Ann. Sc. Norm. Super. Pisa Cl. Sci.}, 4:703--728, 2005.




\bibitem{LinWeiZhang2015}
Chang-Shou Lin, Jun-Cheng Wei, and Lei Zhang.
\newblock Classification of blowup limits for {SU(3)} singular Toda systems.
\newblock {\em Anal. PDE}, 8(4):807--837, 2015.


\bibitem{lin2012classification}
Chang-Shou Lin, Juncheng Wei, and Dong Ye.
\newblock Classification and nondegeneracy of {S}{U}({N}+1) {T}oda system with
singular sources.
\newblock {\em Invent. math.}, 190:169--207, 2012.




\bibitem{Lucia2002Vortex}
Marcello Lucia and Margherita Nolasco.
\newblock {S}{U}(3) {C}hern--{S}imons vortex theory and {T}oda systems.
\newblock {\em J. Differ. Equ.}, 184(2):443--474, 2002.


\bibitem{malchiodi2008morse}
Andrea Malchiodi.
\newblock Morse theory and a scalar field equation on compact surfaces.
\newblock {\em Adv. Differential Equations}, 13(11-12):1109--1129, 2008.


\bibitem{Malchiodi2007SomeER}
Andrea Malchiodi and Cheikh~Birahim Ndiaye.
\newblock Some existence results for the {T}oda system on closed surfaces.
\newblock {\em Rend. Lincei Mat. Appl.}, 18(4):391--412, 2007.




\bibitem{MalchiodiRuiz2013}
Andrea Malchiodi and David Ruiz.
\newblock A variational analysis of the toda system on compact surfaces.
\newblock {\em Comm. Pure Appl. Math.}, 66(3):332--371, 2013.


\bibitem{musso_new_2016}
Monica Musso, Angela Pistoia, and Juncheng Wei.
\newblock New blow-up phenomena for {S}{U}({N}+ 1) {Toda} system.
\newblock {\em J. Differential Equations}, 260(7):6232--6266, 2016.


\bibitem{Ndiaye2008}
Cheikh~Birahim Ndiaye.
\newblock Conformal metrics with constant $q$-curvature for manifolds with
boundary.
\newblock {\em Comm. Anal. Geom.}, 16(5):1049--1124, 2008.


\bibitem{nolasco1999double}
Margherita Nolasco and Gabriella Tarantello.
\newblock Double vortex condensates in the {C}hern-{S}imons theory.
\newblock {\em Calc. Var. Partial Differ. Equ.}, 9:31--94, 1999.


\bibitem{nolasco2000vortex}
Margherita Nolasco and Gabriella Tarantello.
\newblock Vortex condensates for the {S}{U}(3) {C}hern-{S}imons theory.
\newblock {\em Comm. Math. Phys.}, 213(3):599--639, 2000.


\bibitem{struwe1985evolution}
Michael Struwe.
\newblock On the evolution of harmonic mappings of {R}iemannian surfaces.
\newblock {\em Comment. Math. Helv.}, 60(1):558--581, 1985.


\bibitem{struwe2000variational}
Michael Struwe.
\newblock {\em Variational Methods: Applications to Nonlinear Partial
Differential Equations and {H}amiltonian Systems}.
\newblock Springer-Verlag, Berlin, 3rd edition, 2000.


\bibitem{wang_wei2002_steady}
Guofang Wang and Juncheng Wei.
\newblock Steady state solutions of a reaction-diffusion system modeling
{C}hemotaxis.
\newblock {\em Math. Nachr.}, 233-234(1):221--236, 2002.


\bibitem{yang1999relativistic}
Yisong Yang.
\newblock The relativistic non-abelian {C}hern-{S}imons equation.
\newblock {\em Comm. Math. Phys.}, 186(1):199--218, 1999.


\bibitem{yang2001solitons}
Yisong Yang.
\newblock {\em Solitons in Field Theory and Nonlinear Analysis}.
\newblock Springer Monographs in Mathematics. Springer, New York, 2001.


\bibitem{yang2006extremal}
Yunyan Yang.
\newblock Extremal functions for {M}oser-{T}rudinger inequalities on
2-dimensional compact {R}iemannian manifolds with boundary.
\newblock {\em Int. J. Math.}, 17(3):313--330, 2006.


\bibitem{yang_wang2021}
Yunyan Yang and Jie Zhou.
\newblock Blow-up analysis involving isothermal coordinates on the boundary of
compact {Riemann} surfaces.
\newblock {\em J. Math. Anal. Appl.}, 504(2):125440, 2021.


\bibitem{zhu_solutions_2011}
Xiao~Bao Zhu.
\newblock Solutions for {Toda} system on {Riemann} surfaces with boundary.
\newblock {\em Acta Math. Sin.-English Ser.}, 27(8):1501--1520, 2011.


\end{thebibliography}


\end{document}